\documentclass[11pt]{article}
\usepackage{latexsym}
\usepackage{amssymb}
\usepackage{graphicx}
\usepackage{enumerate}
\usepackage{amssymb}
\usepackage{amsmath}
\usepackage{color}
\usepackage{tikz}
\usepackage{hyperref}
\hypersetup{hypertex=true,
	colorlinks=true,
	linkcolor=blue,
	anchorcolor=blue,
	citecolor=blue}
\usepackage{amsthm}
\usepackage{cleveref}
\crefformat{section}{\S#2#1#3} 
\crefformat{subsection}{\S#2#1#3}
\crefformat{subsubsection}{\S#2#1#3}
\def\osc{{\mathrm{osc}}}

\def\l{\lambda}

\def\dvol{{\mathrm{dvol}}}
\def\ep{\epsilon}

\numberwithin{equation}{section}
\theoremstyle{plain}
\newtheorem{thm}{Theorem}[section]
\newtheorem{lem}[thm]{{{L}}emma}
\newtheorem{cor}[thm]{{C}orollary}
\newtheorem{prop}[thm]{{P}roposition}

\newtheorem{defn}[thm]{{D}efinition}

\def\XXint#1#2#3{{\setbox0=\hbox{$#1{#2#3}{\int}$}
		\vcenter{\hbox{$#2#3$}}\kern-.5\wd0}}

\newcommand{\be}{\begin{equation}}
	\newcommand{\ee}{\end{equation}}
\newcommand{\ba}{\begin{aligned}}
	\newcommand{\ea}{\end{aligned}}
\newcommand{\Ric}{\operatorname{Ric}}
\newcommand{\Rm}{\operatorname{Rm}}
\newcommand{\Hess}{\operatorname{Hess}}
\newcommand{\Vol}{\operatorname{Vol}}

\newcommand{\inj}{\operatorname{inj}}
\newcommand{\dd}{\,\mathrm{d}}

\title{The Weyl Law Meets Large-Scale Regularity  on Ricci Shrinkers}
\author{ Junrong Yan\footnote{Department of Mathematics, Northeastern University, Boston, MA, 02215, j.yan@northeastern.edu}
}
\date{}

\begin{document}
	\maketitle
	\begin{abstract}
		
		We prove that the weighted Laplacian, or equivalently its conjugate
		Schr\"odinger operator, on every complete gradient Ricci shrinker satisfies
		the classical Weyl law.
		
	The main difficulty is that uniform bounded geometry is not known for general Ricci shrinkers. To overcome this, we establish a large-scale regularity property for complete gradient Ricci shrinkers and apply it to the spectral asymptotics of the weighted Laplacian. We prove that, inside large geodesic balls of radius $R$, the region where the curvature radius is smaller than  $R^{-1}$ occupies an asymptotically negligible proportion of the volume. The proof uses the Ricci flow associated with the shrinker, together with the curvature-radius estimates and Sobolev inequalities of Li--Wang.

	\end{abstract}
	
	\section{Introduction}

	\subsection{Overview}
	
	Let $(M,g,e^{-f}\dvol_g)$ be a weighted Riemannian manifold.  After a unitary
	conjugation, its weighted Laplacian becomes a Schr\"odinger operator on
	$L^2(M,\dvol_g)$ of the form
	$$
	H=\Delta+V,
	\qquad
	V\in C^\infty(M).
	$$
	When $V\to+\infty$, the natural candidate for the leading asymptotic of its
	eigenvalue counting function is
	\be\label{eq:intro-general-Weyl}
	N_H(\lambda)
	\sim
	(2\pi)^{-n}\omega_n
	\int_M(\lambda-V(x))_+^{n/2}\,\dvol_g(x).
	\ee
	On a noncompact manifold, discreteness of the spectrum alone does not imply
	\eqref{eq:intro-general-Weyl}: the interaction between the potential and the
	geometry at infinity may change even the leading term; see \cite[\S5]{BDY}
	for counterexamples.
	
	For a normalized gradient Ricci shrinker,
	$$
	\Ric+\nabla^2f=\frac12g,
	\qquad
	R+|\nabla f|^2=f,
	$$
	the weighted Laplacian is conjugate to
	$$
	H_f=\Delta+V_f,
	\qquad
	V_f=\frac14(f+R-n).
	$$
	The favorable feature is that the potential is well controlled: the estimates of Cao--Zhou \cite{CZ}, together
	with subsequent refinements, give quadratic growth of $f$ and control of the
	volumes of its sublevel sets.  The spectrum is therefore discrete.  Our main result is that the expected classical asymptotic nevertheless holds,
	even without uniform geometric control of the shrinker:
	\[
	N_f(\lambda)
	\sim
	(2\pi)^{-n}\omega_n
	\int_M\left(\lambda-\frac14(f+R-n)\right)_+^{n/2}\,\dvol_g,
	\qquad \lambda\to+\infty.
	\]

	To explain our approach to establishing the Weyl law, and why we need such
	large-scale regularity, recall that, in joint work with Braverman and Dai
	\cite{BDY}, we established a criterion for the classical Weyl law for
	Schr\"odinger operators on complete Riemannian manifolds.  That criterion balances the local geometry, including curvature
	and injectivity-radius bounds, against the growth and oscillation of the
	potential.  When this balance fails, the Weyl law may fail
	\cite[$\S$5.1, $\S$5.2]{BDY}.  Applied directly here, it would ask for the
	uniform estimate
	\be\label{eq: expected}
	\sqrt{\lambda}\,
	\inf_{\{V_f<C\lambda\}}\rho
	\longrightarrow\infty,
	\ee
	where $\rho$ is the curvature radius defined in \eqref{eq:defn rho}.  Such an
	estimate is not known and is stronger than what the Weyl law requires.
	We instead prove that, for every fixed $A,C>0$,
	\be\label{eq: overview of result}
	\Vol_g\{x\in B_g(p,CR):r_{\Rm}(x)\leq AR^{-1}\}
	=o\Big(
	\Vol_g\big(B_g(p,R)\big)
	\Big),
	\qquad
	R\to\infty.
	\ee
	Here $r_{\Rm}$ is the curvature radius associated with the canonical Ricci
	flow, $p\in M$ is fixed, and $B_g(p,R)$ is the geodesic ball of radius $R$.

	We place this estimate in the context of geometric control for Ricci shrinkers.
	Munteanu--M.-T.~Wang \cite{MunteanuMuTaoWang} established integral curvature
	estimates for gradient shrinkers and proved that bounded Ricci curvature
	forces the Riemann curvature tensor to grow at most polynomially.  In dimension
	four, Munteanu--Wang \cite{MunteanuJiapingWang} proved the much stronger
	estimate
	$
	|\Rm|\leq C R
	$
	under the assumption that the scalar curvature is bounded.  Their work also
	gives detailed information on the geometry at infinity; see further
	\cite{MunteanuWangInfinity}.  Related rigidity results for prescribed
	asymptotically conical geometry were obtained by Kotschwar--Wang
	\cite{KotschwarWang}.  Most recently, Li--Wang proved that every K\"ahler Ricci
	shrinker surface has bounded sectional curvature and obtained a complete
	classification \cite{LiWangKahler}.
	
	A different development treats Ricci shrinkers without imposing curvature
	bounds in advance.  Bamler study $\mathbb F$-compactness and structure theory for
	Ricci flows
	\cite{BamlerStructure,BamlerCompactness}.  In the setting of flows associated
	with shrinkers, Li--Wang established heat-kernel estimates,
	curvature-radius comparison, and quantitative regularity results
	\cite{LW20,LWheat}.  Related structure results for noncollapsed shrinkers and
	their limit spaces were obtained by Li--Li--Wang
	\cite{LiLiWangStructure} and Huang--Li--Wang \cite{HuangLiWang}.  
	
	The spectral geometry of Ricci shrinkers has also developed through the study of their first few eigenvalues.  The Bakry--\'Emery curvature
	identity gives the sharp lower bound
	$\lambda_1\geq\frac12$, and equality is closely tied to Euclidean splitting;
	see Hein--Naber \cite{HeinNaber} and Cheng--Zhou \cite{ChengZhouDrifted}.
	Colding--Minicozzi \cite{ColdingMinicozziGrowth,
		ColdingMinicozziModifiedFlow} developed polynomial-growth estimates for
	drift eigenfunctions and related splitting results.   More recently, Li--Zhang--Zhang
	\cite{LiZhangZhangEigenvalues} and Zhang
	\cite{ZhangEigenvalueSplitting} established quantitative versions of this
	phenomenon, relating eigenvalues close to $1/2$ to almost splitting of the
	shrinker.  He \cite{He} obtained upper and lower estimates for higher
	eigenvalues in terms of the volume growth of $f$-sublevel sets, and related
	estimates for Schr\"odinger and drifted Schr\"odinger operators were recently
	obtained in \cite{ChengConradoPinheiroZhou}. 
	There is also a substantial literature connecting this spectral theory with
	harmonic and holomorphic functions of controlled growth \cite{MaiOuLiouville,WuWuPolynomialGrowth,HeOuHolomorphic,HeOuSections}.
	
	These results concern the relation between individual eigenvalues and the geometry and rigidity of Ricci shrinkers.
	The present work concerns the complementary
	part.  We show that the full eigenvalue counting function has the
	classical leading asymptotic on every complete shrinker.

	\subsection{Notation and proof strategy}
	
	Let $(M^n,g,f)$ be a complete gradient shrinking Ricci soliton normalized by
	\be\label{eq:shrinker}
	\Ric_g+\Hess_g f=\frac12g,
	\qquad
	R_g+|\nabla f|_g^2=f.
	\ee
	Here $\Ric_g$, $R_g$, $|\cdot|_g$, and $\Hess_g$ denote the Ricci curvature
	tensor, the scalar curvature, the pointwise norm, and the Hessian induced by
	$g$, respectively.  We use $\Vol_g$, $\dvol_g$, and $d_g$ for the
	corresponding volume, volume measure, and distance, and omit the subscript
	$g$ when no confusion can arise.
	
	We use the convention that the Laplace--Beltrami operator
	$\Delta=-\operatorname{div}\nabla$ is nonnegative.  The weighted Laplacian on
	$L^2(M,e^{-f}\dvol_g)$ is
	\[
	L_f=\Delta+\nabla_{\nabla f}.
	\]
	Conjugation by multiplication with $e^{-f/2}$ transforms $L_f$ into the
	Schr\"odinger operator
	\be\label{eq:intro Hf}
	H_f=\Delta+V_f,
	\qquad
	V_f=\frac14(f+R-n);
	\ee
	see \eqref{eq:conjugation}.  The spectrum of $L_f$, or equivalently of
	$H_f$, is discrete; see, for example,
	\cite{HeinNaber,ChengZhouDrifted}.  We write its eigenvalues, repeated
	according to multiplicity, as
	\[
	0=\lambda_0\leq\lambda_1\leq\cdots
	\]
	and set
	\[
	N_f(\lambda):=\#\{j:\lambda_j<\lambda\}.
	\]
	The corresponding classical phase-space volume is
	\[
	\Phi_f(\lambda)
	:=
	(2\pi)^{-n}\omega_n
	\int_M(\lambda-V_f)_+^{n/2}\,\dvol_g,
	\]
	where $\omega_n$ denotes the volume of the Euclidean unit ball and
	$x_+:=\max\{x,0\}$.  The Weyl law is therefore to determine whether
	\[
	N_f(\lambda)\sim\Phi_f(\lambda),
	\qquad \lambda\to+\infty.
	\]
	We next introduce the geometric scale used in the statement of the main result
	below.
	For $x\in M$, define
	\be\label{eq:defn rho}
	\begin{aligned}
		\rho(x):=\sup\Bigl\{r\in(0,+\infty)\ \Big|\ 
		& r^{2+k}|\nabla^k\Rm_g|(y)\leq1,\quad k=0,1,2\\
		& \inj_g(y)>r,
		\qquad \forall\,y\in B_g(x,r)
		\Bigr\}.
	\end{aligned}
	\ee
	Thus $\rho(x)$ records a scale on which the curvature is controlled and on which the injectivity radius
	does not collapse.  The injectivity-radius condition is essential for the
	Weyl asymptotics; without it, counterexamples occur even when the
	curvature is controlled; see \cite[\S5.1]{BDY}.
	
	We also introduce the sublevel-volume functions
	\[
	F(\lambda):=\Vol_g\{f<\lambda\},
	\qquad
	\sigma(\lambda):=\Vol_g\{V_f<\lambda\}.
	\]
	As explained in the overview, a direct application of the
	criterion  in \cite{BDY} in shrinkers' setting then would require, for every fixed $C>0$,
	\be\label{eq:expected to hold}
	\sqrt{\lambda}\,
	\inf_{x\in\{V_f<C\lambda\}}\rho(x)
	\longrightarrow+\infty.
	\ee
	Such a uniform estimate is not known for general complete Ricci shrinkers. The
	geometric input replaces \eqref{eq:expected to hold} with the following. For every fixed $A,C>0$, we prove
	\[
	\Vol_g\bigl\{
	V_f<C\lambda,\ \sqrt{\lambda}\rho\leq A
	\bigr\}
	=
	o\bigl(\sigma(\lambda)\bigr).
	\]
	The above estimate is equivalent to \eqref{eq: overview of result}.  In low
	dimensions it follows directly from the curvature-radius
	integral estimate of Li--Wang together with the growth of $F(\lambda)$.
	In higher dimensions we combine Li--Wang's spacetime curvature-radius
	estimate with their global Sobolev inequality on Ricci shrinkers.  The details are given in
	\Cref{sec:mainprooflow} and \Cref{sec: bad set}.

	\subsection{Main results}
	
	The main result of the paper is the following Weyl law.
	
	\begin{thm}[Weyl law]
		\label{thm:main}
		Let $(M^n,g,f)$ be a complete gradient Ricci shrinker. Then
		\be\label{eq:Weyl law with scalar curvature}
		N_f(\lambda)
		\sim
		(2\pi)^{-n}\omega_n
		\int_M(\lambda-V_f)_+^{n/2}\,\dvol_g,
		\qquad \lambda\to+\infty.
		\ee
	\end{thm}
	
	The key geometric input is the large-scale regularity:
	
	\begin{thm}[Large-scale regularity]
		\label{thm:estimate}
		Let $(M^n,g,f)$ be a complete gradient Ricci shrinker.  Then, for every fixed
		$A,C>0$,
		\[
		\Vol_g\bigl\{
		f<C\lambda,\ \sqrt{\lambda}\rho\leq A
		\bigr\}
		=
		o\bigl(F(\lambda)\bigr),
		\qquad \lambda\to+\infty.
		\]
		By \Cref{prop:radius-comparison}, \eqref{eq:doubling} in \Cref{prop:coarea},
		and \cite[Theorem 1.1]{CZ}, this is equivalent to
		\[
		\Vol_g\{x\in B_g(p,CR):r_{\Rm}(x)\leq A R^{-1}\}
		=o\Big(
		\Vol_g\big(B_g(p,R)\big)
		\Big),
		\qquad
		R\to\infty.
		\]
		Here $p\in M$ is fixed and $B_g(p,R)$ denotes the geodesic ball of radius
		$R$ centered at $p$.
	\end{thm}
	
	In \Cref{sec:analytic}, we prove  that
	\Cref{thm:estimate} implies \Cref{thm:main}; see
	\Cref{thm:averaged}.

	The scalar curvature appearing in
	$
	V_f=\frac14(f+R-n)
	$
	makes the phase integral in \eqref{eq:Weyl law with scalar curvature}
	less explicit.  To remove it from the leading term, define
	\[
	\widetilde V_f:=\frac14(f-n)
	\quad
	\text{and}\quad
	\Psi_f(\lambda)
	:=
	(2\pi)^{-n}\omega_n
	\int_M(\lambda-\widetilde V_f)_+^{n/2}\,\dvol_g.
	\]
	We prove in \Cref{prop:phase-comparison} that, for every complete noncompact
	Ricci shrinker,
	\be\label{eq:phase-comparison0}
	\Phi_f(\lambda)\sim\Psi_f(\lambda),
	\qquad \lambda\to+\infty.
	\ee
	Consequently:
	
	\begin{cor}
		\label{cor:main}
		Let $(M^n,g,f)$ be a complete Ricci shrinker.  Then
		\[
		N_f(\lambda)
		\sim
		(2\pi)^{-n}\omega_n
		\int_M
		\left(\lambda-\frac14(f-n)\right)_+^{n/2}\,\dvol_g,
		\qquad \lambda\to+\infty.
		\]
	\end{cor}
	
	Thus the leading spectral asymptotic depends only on the volume distribution
	of the shrinker potential.  In particular, when the sublevel volume has a
	precise power-law asymptotic, the phase integral can be evaluated explicitly.
	
	\begin{cor}
		\label{cor:constant}
		Suppose that there exist constants $C_f>0$ and $d>0$ such that
		\be\label{eq:regular-variation}
		F(\lambda)\sim C_f\lambda^{d/2},
		\qquad \lambda\to+\infty.
		\ee
		Then
		\[
		N_f(\lambda)
		\sim
		(2\pi)^{-n}\omega_n C_f4^{d/2}
		\frac{
			\Gamma(\frac d2+1)\Gamma(\frac n2+1)
		}{
			\Gamma(\frac{n+d}{2}+1)
		}
		\lambda^{(n+d)/2},
		\qquad \lambda\to+\infty.
		\]
	\end{cor}
	
	\begin{proof}
		Put $p=n/2$ and $\widetilde\lambda=\lambda+n/4$.  By Tonelli's theorem,
		\[
		\int_M(\widetilde\lambda-f/4)_+^p\,\dvol_g
		=
		p\int_0^{\widetilde\lambda}
		(\widetilde\lambda-u)^{p-1}F(4u)\,du
		=p\widetilde\lambda^pF(4\widetilde\lambda)
		\int_0^1
		(1-v)^{p-1}
		\frac{F(4\widetilde\lambda v)}
		{F(4\widetilde\lambda)}
		\,dv.
		\]
		By \eqref{eq:regular-variation},
		$
		\frac{F(4\widetilde\lambda v)}
		{F(4\widetilde\lambda)}
		\longrightarrow v^{d/2},\l\to\infty
		$
		for every $v\in(0,1]$.  Since the quotient is bounded by $1$, dominated
		convergence gives
		\[
		\int_0^1
		(1-v)^{p-1}
		\frac{F(4\widetilde\lambda v)}
		{F(4\widetilde\lambda)}  
		\,dv
		\longrightarrow
		\int_0^1(1-v)^{p-1}v^{d/2}\,dv
		=
		\frac{
			\Gamma(\frac d2+1)\Gamma(p+1)
		}{
			\Gamma(\frac d2+p+1)
		}.
		\]
		Finally,
		\[
		F(4\widetilde\lambda)
		\sim
		C_f4^{d/2}\widetilde\lambda^{d/2},
		\qquad
		\widetilde\lambda\sim\lambda,
		\]
		which proves the claim.
	\end{proof}
	\subsection{Organization of the paper}
	In \Cref{sec:mainprooflow}, we first prove \Cref{thm:estimate}
	in dimensions $n\leq 6$.  In this range, the curvature-radius integral
	estimate of Li--Wang, together with the lower volume growth of the
	$f$-sublevel sets, is already sufficient to obtain the required
	relative-volume estimate.
	
	In \Cref{sec: review}, we introduce the canonical Ricci flow, the
	spacetime potential, heat-kernel neighborhoods, and the curvature radius,
	and collect all Li--Wang estimates used later.
	
	In \Cref{sec: bad set}, we prove \Cref{thm:estimate} in higher
	dimensions.  The main ingredients are a  curvature-radius
	estimate, a volume-noncollapsing argument from small
	scalar-curvature energy, and a covering argument.
	
	Finally, in \Cref{sec:analytic}, we prove \Cref{thm:averaged}.
	After establishing the required sublevel-set and heat-kernel estimates, we
	prove a uniform local Weyl asymptotic on the geometrically regular region
	and show that the complementary region makes a negligible contribution.
	This  together with
	\Cref{thm:estimate} yields \Cref{thm:main}.
	
	\subsection*{Acknowledgment}
	The author would like to thank Fei He for helpful discussions and for
	raising the question of whether the classical Weyl law holds for the
	weighted Laplacian on every complete Ricci shrinker, which motivated
	the present work.

		\subsection*{AI use disclosure}
	Generative AI tools were used to assist with literature searches, understanding some cited references, and searching proofs and filling in some technical details based on ideas and proof strategies provided by the author, as well as with language editing. The mathematical arguments and proofs were developed and finalized by the author, who take full responsibility for the content of the manuscript.

	\section{Proof of \Cref{thm:estimate}  for $n\leq6$}
	\label{sec:mainprooflow}
	
	The main input is the following inequality, which follows from
	\cite[Theorem 1.7]{LWheat} and \Cref{prop:radius-comparison} in \Cref{sec: review}. Let $p$ be a fixed point. Then, for every $\epsilon>0$,
	\be\label{eq:LW-weighted}
	\int_{\{d(p,x)\geq1\}}
	\frac{\rho(x)^{-4+2\epsilon}}
	{d(p,x)^{n-2+2\epsilon}}\dd\mathrm{vol}_g(x)<\infty.
	\ee

	\begin{lem}\label{lem:bad-set}
		Fix $\epsilon\in(0,1)$ and $A,C>0$.  Then
		\[
		\Vol_g\{f<C\lambda,\ \sqrt\lambda\,\rho\leq A\}
		=o\left(\lambda^{(n-6+4\epsilon)/2}\right).
		\]
	\end{lem}
	
	\begin{proof}
		Put
		\[
		q=4-2\epsilon,\qquad
		I(R)=\int_{B_g(p,R)}\rho^{-q}\dd\mathrm{vol}_g.
		\]
		The quadratic lower growth of $f$ (c.f. \cite[Theorem 1.1]{CZ}) gives
		$
		\{f<C\lambda\}\subset B_g(p,C'\sqrt\lambda)
		$
		for all large $\lambda$.  
		Chebyshev's inequality gives
		\begin{align}\label{eq: Che}
			\Vol_g\{f<C\lambda,\sqrt\lambda\,\rho\leq A\}
			&\leq (C''A\lambda^{-1/2})^q I(C'\sqrt\lambda)
		\end{align}
		We next prove
		\be\label{eq:I-little-o}
		I(\l)=o(\l^{n-2+2\epsilon}).
		\ee
		The integral over $B_g(p,1)$ is finite.  On its complement,
		\[
		\frac1{\l^{n-2+2\epsilon}}
		\int_{\{1<d(p,x)<\l\}}\rho^{-q}\dd\mathrm{vol}_g
		=
		\int_{\{1<d(p,x)<\l\}}
		\left(\frac{d(p,x)}\l\right)^{n-2+2\epsilon}
		\frac{\rho^{-q}}{d(p,x)^{n-2+2\epsilon}}\dd\mathrm{vol}_g.
		\]
		For each fixed $x$, the integrand tends to zero as $\l\to\infty$, and it is
		dominated by the integrable function in \eqref{eq:LW-weighted}.
		Dominated convergence proves \eqref{eq:I-little-o}.

		The Lemma then follows from \eqref{eq: Che} and \eqref{eq:I-little-o}.
		
	\end{proof}

	\begin{proof}[Proof of Theorem~\ref{thm:estimate} for $n\leq6$]
		We may assume that $M$ is noncompact. It follows from \cite[Theorem 1.6]{MunteanuWang2012} and \cite[Theorem 1.1]{CZ} that, for some $C_0>0$ and all sufficiently large $\l$,
		\be\label{eq:lower-growth}
		F(\l)\geq C_0\sqrt{\l}.
		\ee
		Fix $\epsilon\in (\max\{0,1-n/2\},\min\{1,7/4-n/4\})$. For fixed $A,C>0$, Lemma~\ref{lem:bad-set} and \eqref{eq:lower-growth} give
		\begin{align*}
			\frac{\Vol_g\{f<C\lambda,\sqrt\lambda\,\rho\leq A\}}
			{F(\lambda)}
			&=
			o\left(\lambda^{(n-6+4\epsilon-1)/2}\right)
			\longrightarrow0.
		\end{align*}
		Thus, \Cref{thm:estimate} follows for $n\leq 6$. 
	\end{proof}
	
	In higher dimensions, the analysis is more complicated and requires Li--Wang's  analysis for the Ricci flows induced by Ricci shrinkers.

	\def\badsetone{1}
	\section{Ricci-flow preliminaries and local estimates}\label{sec: review}
	
	We collect here the Ricci-flow notation and local estimates used in
	\Cref{sec: bad set}. The principal inputs are the heat-kernel,
	curvature-radius, and Sobolev estimates of Li--Wang
	\cite{LW20,LWheat}, together with some notions considered in
	\cite{BamlerCompactness,BamlerStructure}.
	
	\subsection{The canonical flow and the spacetime potential}
	
	Let $(M^n,g,f)$ be a normalized gradient Ricci shrinker and fix
	$\Lambda<\infty$ such that
	\[
	\boldsymbol\mu(g):=
	\log\left((4\pi)^{-n/2}\int_M e^{-f}\,\dvol_g\right)\ge -\Lambda.
	\]
	The diffeomorphisms
	\[
	\frac{d}{dt}\psi^t=\frac{\nabla f\circ\psi^t}{1-t},
	\qquad
	\psi^0=\mathrm{id},
	\]
	define the associated Ricci flow
	\[
	g(t)=(1-t)(\psi^t)^*g,\qquad -\infty<t<1.
	\]
	It satisfies
	\[
	\partial_tg(t)=-2\Ric_{g(t)},\qquad g(0)=g.
	\]
	A time translation and parabolic rescaling of this flow will be called
	a Ricci flow \emph{induced} by the shrinker.
	
	Put
	\[
	\bar\tau:=1-t,\qquad
	f_t:=(\psi^t)^*f,\qquad
	\mathcal F(\cdot,t):=\bar\tau f_t.
	\]
	With our convention that $\Delta_{g(t)}$ is nonnegative, the shrinker
	identities give
	\begin{align}
		\partial_t\mathcal F&=-\bar\tau R_{g(t)}, \label{eq:F-time-derivative}\\
		\bar\tau R_{g(t)}-\Delta_{g(t)}\mathcal F&=\frac n2, \label{eq:F-trace}\\
		\bar\tau^2R_{g(t)}+|\nabla\mathcal F|_{g(t)}^2&=\mathcal F, \label{eq:F-gradient}\\
		(\partial_t+\Delta_{g(t)})\mathcal F&=-\frac n2. \label{eq:F-heat}
	\end{align}
	In particular $R\ge0$ and \eqref{eq:F-gradient} imply
	\begin{equation}\label{eq:F-scalar-bound}
		0\le \bar\tau^2R_{g(t)}\le \mathcal F.
	\end{equation}
	
	We will repeatedly use the following coarse comparison; see
	\cite[Lemma~2.3 and Lemma~4.8]{LWheat}.
	
	\begin{prop}
		\label{prop:spacetime-potential-distance-control}
		Let $p$ be a minimum point of $f$, and let
		$J\Subset(-\infty,1)$ be compact.  There is
		$C_J=C_J(n,J)>1$ such that, for all $x\in M$ and $s,t\in J$,
		\begin{align}
			C_J^{-1}\bigl(d_{g(t)}(p,x)-C_J\bigr)_+^2
			&\le \mathcal F(x,t)
			\le C_J\bigl(d_{g(t)}(p,x)+1\bigr)^2,
			\label{eq:23-F-quadratic}\\
			C_J^{-1}\bigl(d_{g(s)}(p,x)+1\bigr)
			&\le d_{g(t)}(p,x)+1
			\le C_J\bigl(d_{g(s)}(p,x)+1\bigr).
			\label{eq:23-time-distance-comparison}
		\end{align}
		The corresponding estimates for an induced flow follow by time
		translation and parabolic rescaling.
	\end{prop}

	\subsection{Heat kernels and parabolic neighborhoods}
	
	Let $H(x,t;y,s)$, $s<t$, denote the heat kernel of the Ricci flow.
	Thus
	\[
	\bigl(\partial_t+\Delta_{g(t)}\bigr)u=0,
	\qquad
	u(x,t)=\int_M H(x,t;y,s)u(y,s)\,\dvol_{g(s)}(y).
	\]
	For a fixed spacetime point $(x,t)$,
	\[
	d\nu_{x,t;s}(y):=
	H(x,t;y,s)\,\dvol_{g(s)}(y)
	\]
	is the conjugate heat-kernel probability measure on the time-$s$
	slice, and $\nu_{x,t;t}=\delta_x$.  Existence, positivity, and
	stochastic completeness for shrinker flows are proved in
	\cite[Theorem~7]{LW20}.
	
	\begin{defn}[Coupling of measures]
		Let $X_1$ and $X_2$ be measurable spaces, and let
		$\mu_1\in\mathcal P(X_1)$ and $\mu_2\in\mathcal P(X_2)$ be
		probability measures.  A \emph{coupling} of $\mu_1$ and $\mu_2$
		is a probability measure
		\[
		q\in\mathcal P(X_1\times X_2)
		\]
		whose first and second marginals are $\mu_1$ and $\mu_2$,
		respectively.  Equivalently,
		\[
		(\pi_1)_*q=\mu_1,
		\qquad
		(\pi_2)_*q=\mu_2,
		\]
		where
		$ \pi_i:X_1\times X_2\longrightarrow X_i
		$
		are the coordinate projections.  The set of all couplings of
		$\mu_1$ and $\mu_2$ is denoted by
		\[
		\operatorname{Cpl}(\mu_1,\mu_2).
		\]
		
	\end{defn}
	
	For probability measures $\mu,\nu$ on a metric space $(X,d)$, write
	\[
	W_1^d(\mu,\nu):=
	\inf_{q\in\operatorname{Cpl}(\mu,\nu)}
	\int_{X\times X}d(x,y)\,dq(x,y).
	\]
	For a Riemannian metric $g$ we abbreviate $W_1^g:=W_1^{d_g}$.
	
	For $S,T^-,T^+\ge0$, define
	\[
	P(x_0,t_0;S,-T^-,T^+)
	:=
	B_{g(t_0)}(x_0,S)
	\times([t_0-T^-,t_0+T^+]\cap I).
	\]
	Set
	\[
	P(x_0,t_0;r):=P(x_0,t_0;r,-r^2,r^2).
	\]
	The heat-kernel neighborhood is
	\begin{align}
		&P^*(x_0,t_0;S,-T^-,T^+) :=
		\left\{
		(x,t):
		\begin{array}{l}
			t\in[t_0-T^-,t_0+T^+]\cap I,\\[1mm]
			W_1^{g(t_0-T^-)}
			\bigl(
			\nu_{x_0,t_0;t_0-T^-},
			\nu_{x,t;t_0-T^-}
			\bigr)<S
		\end{array}
		\right\}.
		\label{eq:23-Pstar-definition}
	\end{align}
	We abbreviate
	\[
	P^*(x_0,t_0;r):=P^*(x_0,t_0;r,-r^2,r^2).
	\]

	\subsection{Curvature radius and Li--Wang estimates}
	
	For a Ricci flow $(M,g(t))_{t\in I}$ define
	\begin{equation}\label{eq:23-curvature-radius}
		r_{\Rm,g}(x,t)
		:=
		\sup\bigl\{
		r>0:
		|\Rm_{g(s)}|(y)\le r^{-2}
		\text{ for all }(y,s)\in P(x,t;r)
		\bigr\}.
	\end{equation}
	For the associated flow at time zero, write
	\begin{equation}\label{eq:23-time-zero-curvature-radius}
		r_{\Rm}(x):=r_{\Rm,g}(x,0).
	\end{equation}

	The following three estimates are the only quantitative results from
	Li--Wang that enter the high-dimensional argument.
	
	\begin{prop}[Theorem~1.2 in \cite{LWheat}]
		\label{prop:LW-volume-estimates}
		Whenever $[t-r^2,t]\subset I$,
		\[
		\Vol_{g(t)}B_{g(t)}(z,r)\leq C(n)r^n.
		\]
		If also $R_{g(t)}\leq r^{-2}$ on $B_{g(t)}(z,r)$, then
		\[
		\Vol_{g(t)}B_{g(t)}(z,r)\geq c(n)e^{-\Lambda}r^n.
		\]
	\end{prop}
	\begin{proof}
		For the associated Ricci flow, these are the non-expanding and
		no-local-collapsing estimates in
		\cite[Theorem~1.2]{LWheat}.  The statement is invariant under time
		translation and parabolic rescaling.
	\end{proof}

	\begin{prop}[Theorem~6.23 in \cite{LWheat}]
		\label{prop:LW-codimension-four}
		For every $q\in(0,4)$, there is $C=C(n,\Lambda,q)<\infty$ such that,
		whenever $P^*(x_0,t_0;r)$ is defined,
		\[
		\int_{[t_0-r^2,t_0+r^2]\cap I}
		\int_{P^*(x_0,t_0;r)_t}r_{\Rm,g}(y,t)^{-q}\,\dvol_{g(t)}(y)\,dt
		\leq Cr^{n+2-q}.
		\]
		Here $P^*(x_0,t_0;r)_t=\{y:(y,t)\in P^*(x_0,t_0;r)\}$.
	\end{prop}
	\begin{proof}
		For the associated flow, take $q=4-2\epsilon$ in
		\cite[Theorem~6.23]{LWheat}.  The general induced-flow statement follows
		by time translation and parabolic rescaling.
	\end{proof}
	
	\begin{prop}[ Proposition~5.4 in \cite{LWheat}]
		\label{prop:LW-center-estimate}
		There is $C_{\mathrm{cen}}=C_{\mathrm{cen}}(n,\Lambda)$ such that the following holds.
		Let
		$
		s\leq t_1<t_2,
		[s,t_2]\subset I,
		$
		and suppose that
		\[
		R_{g(t)}(x_0)\leq (t_2-t_1)^{-1}
		\qquad
		\text{for every }t\in[t_1,t_2].
		\]
		Then
		\[
		W_1^{g(s)}
		\bigl(
		\nu_{x_0,t_2;s},
		\nu_{x_0,t_1;s}
		\bigr)
		\leq
		C_{\mathrm{cen}}\sqrt{t_2-t_1}.
		\]
		In particular,
		\[
		W_1^{g(t_1)}(\nu_{x_0,t_2;t_1},\delta_{x_0})
		\leq C_{\mathrm{cen}}\sqrt{t_2-t_1}.
		\]

	\end{prop}
	\begin{proof}
		For the associated Ricci flow, part~(1) is
		\cite[Proposition~5.4 and Proposition 3.7]{LWheat}.  The assertion is invariant under time
		translation and parabolic rescaling.

	\end{proof}

	\begin{prop}[Comparison of regularity radii]
		\label{prop:radius-comparison}
		For the radius $\rho$ defined in \eqref{eq:defn rho}, there is
		$c=c(n,\Lambda)>0$ such that
		\[
		c\,r_{\Rm}(x)\le \rho(x)\le c^{-1}r_{\Rm}(x),
		\qquad x\in M.
		\]
	\end{prop}
	
	\begin{proof}
		Set $r=r_{\Rm}(x)$.  Shi's estimates \cite{Shi} give
		\[
		r^{2+k}|\nabla^k\Rm|\le C_k,\qquad 0\le k\le2,
		\]
		on a fixed fraction of $B_g(x,r)$.  By
		\Cref{prop:LW-volume-estimates}, for every $s<r/2$,
		\[
		\Vol_g B_g(y,s)\ge c(n)e^{-\Lambda}s^n,
		\qquad y\in B_g(x,r/2).
		\]
		The Cheeger--Gromov--Taylor injectivity estimate \cite{CGT} then gives
		$\inj_g\ge c_0r$ on $B_g(x,r/4)$, and hence
		$\rho(x)\ge c\,r_{\Rm}(x)$.  The reverse inequality follows from
		\cite[Corollary~6.22]{LWheat}.
	\end{proof}
	
	We also use the elementary continuity of the curvature radius.
	
	\begin{lem}\label{lem:rRm-time-continuity}
		Let $(M^n,g(t))_{t\in I}$ be a smooth complete Ricci flow. Then
		\begin{enumerate}[(1)]
			\item for every $t\in I$ and $x,y\in M$,
			\[
			\bigl|r_{\Rm,g}(x,t)-r_{\Rm,g}(y,t)\bigr|
			\le d_{g(t)}(x,y);
			\]
			\item there is $C_n<\infty$ such that, for every fixed $x$,
			\[
			\bigl|
			r_{\Rm,g}(x,t_2)^2-r_{\Rm,g}(x,t_1)^2
			\bigr|
			\le C_n|t_2-t_1|
			\]
			whenever $t_1,t_2$ lie in a compact subinterval of $I$.
		\end{enumerate}
	\end{lem}
	
	\begin{proof}
		Both assertions follow directly from the definition
		\eqref{eq:23-curvature-radius} and the Ricci-flow equation.
	\end{proof}

	\subsection{Fixed rescaled blocks}
	
	For $\lambda\ge1$ set
	\[
	\widehat g_\lambda(t):=\lambda g(\lambda^{-1}t),
	\qquad
	\widehat r_\lambda(x,t):=
	r_{\Rm,\widehat g_\lambda}(x,t).
	\]
	
	\begin{lem}
		\label{lem:rescaled-scalar-bound}
		Fix $C_0,L,T>0$.  There are
		\[
		C_1=C_1(C_0)>C_0,\qquad K=K(C_0)<\infty,
		\]
		and $\lambda_0=\lambda_0(C_0,L,T)<\infty$ such that, whenever
		$\lambda\ge\lambda_0$ and $x\in\{f<C_0\lambda\}$,
		\begin{equation}\label{eq:rescaled-block-bound}
			y\in B_{\widehat g_\lambda(0)}(x,L),\quad |t|\le T
			\quad\Longrightarrow\quad
			\mathcal F(y,\lambda^{-1}t)<C_1\lambda,\qquad
			0\le R_{\widehat g_\lambda(t)}(y)\le K.
		\end{equation}
		Moreover, if $0<T_-\le K^{-1}$, then
		\begin{equation}\label{eq:ordinary-to-Pstar}
			P(x,0;L,-T_-,0)
			\subset
			P^*\bigl(
			x,0;L+C_{\mathrm{cen}}\sqrt{T_-},-T_-,0
			\bigr).
		\end{equation}
	\end{lem}
	
	\begin{proof}
		For $|t|\le T$, the shrinker identities and
		\eqref{eq:F-time-derivative}--\eqref{eq:F-scalar-bound} give, uniformly
		for large $\lambda$,
		\begin{equation}\label{eq:F-time-comparison}
			e^{-C_T/\lambda}f(z)
			\le \mathcal F(z,\lambda^{-1}t)
			\le e^{C_T/\lambda}f(z).
		\end{equation}
		If $y\in B_{\widehat g_\lambda(0)}(x,L)$, then
		$|\nabla\sqrt f|_g\le1/2$ gives
		\[
		\sqrt{f(y)}
		\le\sqrt{f(x)}+\frac{L}{2\sqrt\lambda}
		\le\sqrt{C_0\lambda}+\frac{L}{2\sqrt\lambda}.
		\]
		Thus $f(y)\le C_1'\lambda$ for all sufficiently large $\lambda$, and
		\eqref{eq:F-time-comparison} gives the first assertion in
		\eqref{eq:rescaled-block-bound}.  The scalar bound follows from
		\eqref{eq:F-scalar-bound} after parabolic rescaling.
		
		For \eqref{eq:ordinary-to-Pstar}, let
		$(y,t)\in P(x,0;L,-T_-,0)$.  Since $T_-\le K^{-1}$,
		\Cref{prop:LW-center-estimate} and
		\eqref{eq:rescaled-block-bound} give
		\[
		W_1^{\widehat g_\lambda(-T_-)}
		\bigl(
		\nu_{y,t;-T_-},\nu_{y,0;-T_-}
		\bigr)
		\le C_{\mathrm{cen}}\sqrt{T_-}.
		\]
		The $W_1$-contraction of the heat flow gives
		\[
		W_1^{\widehat g_\lambda(-T_-)}
		\bigl(
		\nu_{x,0;-T_-},\nu_{y,0;-T_-}
		\bigr)
		\le d_{\widehat g_\lambda(0)}(x,y)<L.
		\]
		The triangle inequality proves \eqref{eq:ordinary-to-Pstar}.
	\end{proof}

	\section{Proof of \Cref{thm:estimate} in higher dimensions}
	\label{sec: bad set}
	
	Since \Cref{thm:estimate} was proved in \Cref{sec:mainprooflow} for
	$n\le6$, it remains to treat $n\ge7$.  By
	\Cref{prop:radius-comparison}, it is enough to prove the following
	statement.
	
	\begin{thm}\label{thm: bad set}
		Let $(M^n,g,f)$ be a complete normalized gradient Ricci shrinker with
		$n\ge7$.  Then, for every fixed $A,C_0>0$,
		\begin{equation}\label{eq:relative-rRm}
			\Vol_g
			\bigl\{
			f<C_0\lambda,\ \sqrt{\lambda}\,r_{\Rm}\le A
			\bigr\}
			=o\bigl(F(\lambda)\bigr),
			\qquad \lambda\to+\infty.
		\end{equation}
	\end{thm}
	
	For $\lambda\ge1$ we continue to write
	\[
	\widehat g_\lambda(t):=\lambda g(\lambda^{-1}t),
	\qquad
	\widehat r_\lambda(x,t):=
	r_{\Rm,\widehat g_\lambda}(x,t).
	\]

	Let $(g(t))_{t<1}$ be the canonical Ricci flow associated with the
	shrinker, and set
	\[
	\widehat g_\lambda(t):=\lambda g(\lambda^{-1}t),
	\qquad
	\widehat r_\lambda(x,t)
	:=
	r_{\Rm,\widehat g_\lambda}(x,t).
	\]
	The proof is based on the scalar-curvature energy
	\[
	\mathcal E_{\lambda,L}(x)
	:=
	\int_{B_{\widehat g_\lambda(0)}(x,2L)}
	R_{\widehat g_\lambda(0)}
	\,\dvol_{\widehat g_\lambda(0)},
	\]
	where $L$ is fixed.
	
	We establish two estimates. First, the  curvature-radius
	integral estimate
	of Li--Wang (see Theorem 6.23 in Li--Wang \cite{LWheat} or
	\Cref{prop:LW-codimension-four}) implies,
	that
	\[
	\Vol_{\widehat g_\lambda(0)}
	\left(
	B_{\widehat g_\lambda(0)}(x,L)
	\cap
	\{\widehat r_\lambda(\cdot,0)\leq A\}
	\right)
	\leq
	C A^{q-2}L^{n+2-q},
	\qquad
	2<q<4.
	\]
	See \Cref{lem:slice-tail}.
	
	Second, if $\mathcal E_{\lambda,L}(x)$ is sufficiently small, then
	\Cref{lem:time-zero-noncollapse} gives
	\[
	\Vol_{\widehat g_\lambda(0)}
	B_{\widehat g_\lambda(0)}(x,L/2)
	\geq
	cL^n.
	\]

	To conclude, we cover $\{f<C_0\lambda\}$ by balls centered at a
	maximal $L$-separated set and classify these balls according to
	whether $\mathcal E_{\lambda,L}$ lies above or below a fixed
	threshold, referring to them as high-energy and low-energy balls,
	respectively. The total scalar-curvature energy satisfies
	\[
	\int_{\{ f<C_1\lambda\}}
	R_{\widehat g_\lambda(0)}
	\,\dvol_{\widehat g_\lambda(0)}\,dt
	=
	O\bigl(\lambda^{n/2-1}F(\lambda)\bigr).
	\]
	Consequently, Chebyshev's inequality, together with the bounded
	overlap of the covering, shows that the union of the high-energy
	balls has volume
	\[
	o\bigl(\lambda^{n/2}F(\lambda)\bigr).
	\]
	
	For the low-energy balls, the lower volume estimate controls their
	number, while the upper estimate for the small-curvature-radius set
	gives a relative bound of the form
	$
	A^{q-2}L^{2-q}.
	$
	Since $q>2$, this quantity can be made arbitrarily small by choosing
	$L$ sufficiently large. This proves \eqref{eq:relative-rRm}.

	\subsection{Two volume estimates}
	
	Let $(M^n,g,f)$ be a normalized Ricci shrinker and let
	$(g(t))_{t<1}$ be its canonical Ricci flow.    For $\lambda\geq 1$, set
	\[
	\widehat g_\lambda(t):=\lambda g(\lambda^{-1}t),
	\qquad
	\widehat r_\lambda(x,t):=r_{\mathrm{Rm},\widehat g_\lambda}(x,t)
	=\sqrt\lambda\,
	r_{\mathrm{Rm},g}(x,\lambda^{-1}t).
	\]
	
	Choose once and for all
	\[
	2<q<4.
	\]
	
	\begin{lem}
		\label{lem:slice-tail}
		Fix $A,C_0>0$.  There exists
		\[
		C_{\mathrm{sm}}
		=
		C_{\mathrm{sm}}(n,\Lambda,p,A,C_0)<\infty
		\]
		such that, for every $L\ge\max\{A,1\}$ and all sufficiently large
		$\lambda$, if $x\in\{f<C_0\lambda\}$ and $0<\sigma\le A$, then
		\begin{equation}\label{eq:fixed-time-tail}
			\Vol_{\widehat g_\lambda(0)}
			\left(
			B_{\widehat g_\lambda(0)}(x,L)
			\cap
			\{\widehat r_\lambda(\cdot,0)\le\sigma\}
			\right)
			\le
			C_{\mathrm{sm}}\sigma^{q-2}L^{n+2-q}.
		\end{equation}
	\end{lem}
	
	\begin{proof}

		Let  $L\geq\max\{A,1\}$ and $\sigma\leq A$. Set
		\[
		E
		:=
		B_{\widehat g_\lambda(0)}(x,L)
		\cap
		\left\{
		\widehat r_\lambda(\cdot,0)
		\leq \sigma
		\right\}.
		\]

		Let $K=K(C_0)$ be the constant in
		\Cref{lem:rescaled-scalar-bound}, and set
		\[
		c_0
		:=
		\min\left\{
		\frac14,\frac{1}{KA^2}
		\right\}.
		\]
		Then on the one hand, for any $(y,t)\in E\times[-c_0\sigma^2,0]$, by \Cref{lem:rRm-time-continuity},
		\begin{align}
			\widehat r_\lambda(y,t)
			&\leq
			C_{\mathrm{tc}}
			\left(
			\widehat r_\lambda(y,0)+\sqrt{-t}
			\right)\leq
			C_{\mathrm{tc}}(1+\sqrt{c_0})\sigma
			\leq
			2C_{\mathrm{tc}} \sigma.
			\label{eq:dyadic-time-persistence}
		\end{align}
		On the other hand, since
		$
		c_0\sigma^2\leq K^{-1},
		$ by \Cref{lem:rescaled-scalar-bound}(1),
		\be\label{eq:slice-tail-block-containment}
		P(x,0;L,-c_0\sigma^2,0)
		\subset
		P^*(x,0;r_*),\quad \text{where }r_*=2+L+(1+C_{\mathrm{cen}})\sqrt{c_0}\sigma.
		\ee
		(We may increase $\lambda$ so that
		$
		[-r_*^2,r_*^2]
		$
		is contained in the time interval of the rescaled Ricci flow.)

		It follows from
		\eqref{eq:slice-tail-block-containment} and
		the simple fact $ E\times[-c_0\sigma^2,0]\subset P(x,0;L,-c_0\sigma^2,0)$ that
		\be\label{eq:dyadic-tube-Pstar}
		E\times[-c_0\sigma^2,0]
		\subset P^*(x,0;r_*).
		\ee
		Applying \Cref{prop:LW-codimension-four} 
		gives
		\be\label{eq:LW-spacetime-used}
		\int_{-r_*^2}^{r_*^2}
		\int_{P^*(x,0;r_*)_t}
		\widehat r_\lambda(y,t)^{-q}
		\,\dvol_{\widehat g_\lambda(t)}(y)\,dt
		\leq
		C_qr_*^{n+2-q}.
		\ee
		Since $R_{\widehat g_\lambda(t)}\geq0$, the Ricci-flow volume
		equation gives
		\be\label{eq:backward-volume-monotonicity}
		\dvol_{\widehat g_\lambda(t)}
		\geq
		\dvol_{\widehat g_\lambda(0)},
		\qquad t\leq0.
		\ee
		Using \eqref{eq:dyadic-time-persistence},
		\eqref{eq:dyadic-tube-Pstar}, and
		\eqref{eq:backward-volume-monotonicity}, we obtain
		\begin{align*}
			C_qr_*^{n+2-q}
			&\geq
			\int_{-c_0\sigma^2}^{0}
			\int_{E}
			\widehat r_\lambda(y,t)^{-q}
			\,\dvol_{\widehat g_\lambda(t)}(y)\,dt
			\\&\geq
			(4C_{\mathrm{tc}})^{-q}
			\sigma^{-q}
			\int_{-c_0\sigma^2}^{0}
			\Vol_{\widehat g_\lambda(0)}(E)\,dt
			\notag\\
			&=
			c_0(4C_{\mathrm{tc}})^{-q}\sigma^{2-q}
			\Vol_{\widehat g_\lambda(0)}(E).
		\end{align*}
		By the construction of $r_*$ in \eqref{eq:slice-tail-block-containment}, the result follows.
	\end{proof}
	
	\begin{lem}
		\label{lem:time-zero-noncollapse}
		Fix $C_0>0$.  There exist
		\[
		\eta_0=\eta_0(n,\Lambda,C_0)>0,
		\qquad
		c_0=c_0(n,\Lambda)>0
		\]
		with the following property.  For every fixed $L\ge1$, for all
		sufficiently large $\lambda$, if
		\[
		x\in\{f<C_0\lambda\}
		\]
		and
		\begin{equation}\label{eq:integrate-bound-of-scalar}
			\int_{B_{\widehat g_\lambda(0)}(x,2L)}
			R_{\widehat g_\lambda(0)}
			\,\dvol_{\widehat g_\lambda(0)}
			\le\eta_0,
		\end{equation}
		then
		\[
		\Vol_{\widehat g_\lambda(0)}
		B_{\widehat g_\lambda(0)}(x,L/2)
		\ge c_0L^n.
		\]
	\end{lem}
	
	\begin{proof}
		By \eqref{eq:rescaled-block-bound}, there exists $K_1=K_1(C_0)$ such that for all sufficiently large $\l$,
		\be
		0\le R_{\widehat g_\lambda(t)}\le K_1,
		\qquad
		\quad\text{on } B_{\widehat g_\lambda(0)}(x,2L).
		\label{4.14}
		\ee
		The Sobolev inequality of Li--Wang (\cite[Theorem 1.1]{LW20}) is invariant under constant
		rescaling and gives for some $S=S(n,\Lambda)$,
		\be
		\|u\|_{\frac{2n}{n-2}}^2
		\le
		S\left(
		4\int_M|\nabla u|_{\widehat g_\l(0)}^2\,\dvol_{\widehat g_\l(0)}
		+
		\int_MR_{\hat{g}_\l(0)}\,u^2\,\dvol_{\widehat g_\l(0)}
		\right),
		\qquad
		u\in C_c^\infty(M).
		\label{4.15}
		\ee

		Take $u\in C_c^\infty(B_{\widehat g_\lambda(0)}(x,2L))$. By H\"older inequality
		\be\label{eq: holder}
		\int_{B_{\widehat g_\lambda(0)}(x,2L)}R_{\hat{g}_\l(0)}\,u^2\,\dvol_{\widehat g_\l(0)}
		\le
		\|R_{\hat{g}_\l(0)}\|_{L^{n/2}({B_{\widehat g_\lambda(0)}(x,2L)})}
		\|u\|_{\frac{2n}{n-2}}^2.
		\ee
		Moreover, by \eqref{4.14}, \eqref{eq:integrate-bound-of-scalar}  and choosing $
		\eta_0
		:=
		(2S)^{-n/2}K_1^{1-n/2}.
		$
		\be\label{eq: L n/2 bound}
		S^{n/2}\int_{B_{\widehat g_\lambda(0)}(x,2L)}R_{\hat{g}_\l(0)}^{n/2}\,\dvol_{\widehat g_\l(0)}
		\le
		K_1^{n/2-1}
		\int_{B_{\widehat g_\lambda(0)}(x,2L)}R_{\hat{g}_\l(0)}\,\dvol_{\widehat g_\l(0)}\leq \frac{1}{2^n}.
		\ee

		By \eqref{eq: holder}, \eqref{eq: L n/2 bound} and absorbing the scalar-curvature term in \eqref{4.15} yields
		\be
		\|u\|_{\frac{2n}{n-2}}^2
		\le
		8S\int_{B_{\widehat g_\lambda(0)}(x,2L)}|\nabla u|^2\,\dvol_{\widehat g_\l(0)}
		,
		\qquad u\in C_c^\infty({B_{\widehat g_\lambda(0)}(x,2L)}).
		\ee
		It is standard that the inequality above implies 
		\[
		\Vol_{\widehat g_\l(0)} B_{\widehat g_\l(0)}(x,r)
		\ge c(n,S)r^n,\qquad 0<r\le L/2,
		\] see
		\cite[Theorem~3.1.5]{SaloffCoste2002}.
	\end{proof}

	\subsection{Proof of \Cref{thm: bad set}}
	
	We are now ready to prove \Cref{thm: bad set}.

	By \Cref{lem:slice-tail}, for every fixed
	\(L\geq \max\{A,1\}\), every \(x\in\{f<C_0\lambda\}\), and all sufficiently
	large \(\lambda\),
	\begin{equation}\label{eq:fixed-time-small-rm}
		\Vol_{\widehat g_\lambda}
		\Big(
		B_{\widehat g_\lambda}(x,L)
		\cap\{\widehat r_\lambda\leq A\}
		\Big)
		\leq
		C_{\mathrm{sm}} A^pL^{n-p},
	\end{equation}
	where \(C_{\mathrm{sm}}=C_{\mathrm{sm}}(n,\Lambda,A,p,C_0)\).
	
	Let \(\eta_0=\eta_0(n,\Lambda,C_0)>0\) be the constant in
	Lemma~\ref{lem:time-zero-noncollapse}. For \(x\in\{f<C_0\lambda\}\), define
	the energy
	\begin{equation}\label{eq:block-energy}
		\mathcal E_{\lambda,L}(x)
		:=
		\int_{B_{\widehat g_\lambda(0)}(x,2L)}
		R_{\widehat g_\lambda(0)}
		\,\dvol_{\widehat g_\lambda(0)}.
	\end{equation}
	
	Choose a maximal \(L\)-separated set
	$
	\{x_i\}_{i=1}^{N_\lambda}
	\subset \{f<C_0\lambda\}
	$
	with respect to \(\widehat g_\lambda(0)\). Then
	$
	\{f<C_0\lambda\}
	\subset
	\bigcup_i B_{\widehat g_\lambda(0)}(x_i,L),
	$
	and the balls
	$
	B_{\widehat g_\lambda(0)}(x_i,L/2)
	$
	are pairwise disjoint.
	
	We next verify that the integral domains appearing in
	\eqref{eq:block-energy}, associated with the points \(x_i\), have uniformly
	bounded overlap. Recall that \(L\) is fixed. By
	\Cref{lem:rescaled-scalar-bound}(1), there exist
	\(K_1=K_1(C_0)>0\) and \(C_1=C_1(C_0)>0\) such that
	\begin{equation}\label{eq:expanded-sublevel-scalar-bound}
		B_{\widehat g_\lambda(0)}(x_i,2L)
		\subset
		\Big\{
		z:
		R_{\widehat g_\lambda(0)}(z)\leq K_1,\;
		f(z)<C_1\lambda
		\Big\},
	\end{equation}
	provided that \(\lambda\) is sufficiently large.
	
	Choose the fixed radius
	\begin{equation}\label{eq:overlap-small-radius}
		r_0
		:=
		\min\Big\{
		4^{-1},\,
		K_1^{-1/2}
		\Big\}.
	\end{equation}
	By \Cref{prop:LW-volume-estimates},
	\begin{equation}\label{eq:overlap-small-ball-lower-bound}
		\Vol_{\widehat g_\lambda(0)}
		B_{\widehat g_\lambda(0)}(x_i,r_0)
		\geq
		c(n)e^{-\Lambda}r_0^n.
	\end{equation}
	Since \(L\geq1\) and the points \(x_i\) are \(L\)-separated, the balls
	$
	B_{\widehat g_\lambda(0)}(x_i,r_0)
	$
	are pairwise disjoint. The volume upper bound in
	\Cref{prop:LW-volume-estimates} then implies that every point is contained in
	at most \(M_L\) of the relevant enlarged balls, where
	$
	M_L
	:=
	\frac{C(n)(2L+r_0)^n}{c(n)e^{-\Lambda}r_0^n}.
	$
	
	We next estimate the total energy
	\(\sum_i\mathcal{E}_{\lambda,L}(x_i)\). By the uniformly bounded overlap and
	the estimate
	$
	\int_{\{f<s\}} R\,\dvol_g
	\leq
	\frac n2 F(s)
	$
	see \cite[\S~3]{CZ} or Proposition~\ref{prop:coarea} below, we have
	\begin{align}\label{eq: total energy}
		\sum_i\mathcal{E}_{\lambda,L}(x_i)
		&\leq
		M_L
		\int_{\{f<C_1\lambda\}}
		\widehat R_\lambda\,\dvol_{\widehat g_\lambda(0)}
		\leq
		C'\lambda^{n/2-1}F(\lambda),
	\end{align}
	where \(C'=C'(n,\Lambda,C_0,L)\).
	
	Call \(i\) \emph{high-energy} if
	\[
	\mathcal E_{\lambda,L}(x_i)>\eta_0,
	\]
	and \emph{low-energy} otherwise. For the high-energy indices,
	\eqref{eq: total energy} and Chebyshev's inequality give
	\[
	\#\{i:\text{\(i\) is high-energy}\}
	\leq
	C'\eta_0^{-1}\lambda^{n/2-1}F(\lambda).
	\]
	Since, by \Cref{prop:LW-volume-estimates},
	\[
	\Vol_{\widehat g_\lambda(0)}
	B_{\widehat g_\lambda(0)}(x_i,L)
	\leq
	C(n)L^n,
	\]
	we obtain, after rescaling,
	\begin{equation}\label{eq:high-energy-volume-g}
		\Vol_g
		\Big(
		\bigcup_{i\ {\rm high}}
		B_{\widehat g_\lambda(0)}(x_i,L)
		\Big)
		\leq
		C'C(n)L^{n}\eta_0^{-1}\lambda^{-1}F(\lambda)
		=
		o(F(\lambda)).
	\end{equation}
	
	We now consider the low-energy indices. By
	Lemma~\ref{lem:time-zero-noncollapse},
	\[
	\Vol_{\widehat g_\lambda(0)}
	B_{\widehat g_\lambda(0)}(x_i,L/2)
	\geq
	c_0L^n,
	\]
	where \(c_0=c_0(n,\Lambda)>0\). By
	\eqref{eq:expanded-sublevel-scalar-bound},
	\begin{equation}\label{eq:number-low-energy}
		\#\{i:\text{\(i\) is low-energy}\}\,c_0L^n
		\leq
		\Vol_{\widehat g_\lambda(0)}\{f<C_1\lambda\}
		\leq
		C''\lambda^{n/2}F(\lambda),
	\end{equation}
	where \(C''=C''(n,C_0)\).
	
	Using \eqref{eq:fixed-time-small-rm} and
	\eqref{eq:number-low-energy}, and rescaling, yields
	\begin{equation}\label{eq:low-energy-volume}
		\Vol_g
		\Big(
		\{f<C_0\lambda,\sqrt\lambda\,r_{\Rm}\leq A\}
		\cap
		\bigcup_{i\ {\rm low}}
		B_{\widehat g_\lambda(0)}(x_i,L)
		\Big)
		\leq
		C_{\mathrm{sm}} c_0^{-1}C''A^{q-2}L^{2-q}F(\lambda).
	\end{equation}
	
	Combining \eqref{eq:high-energy-volume-g} and
	\eqref{eq:low-energy-volume}, we obtain, for every fixed
	\(L\geq\max\{A,1\}\),
	\[
	\limsup_{\lambda\to\infty}
	\frac{
		\Vol_g\{f<C_0\lambda,\sqrt\lambda\,r_{\Rm}\leq A\}
	}{F(\lambda)}
	\leq
	C_{\mathrm{sm}} c_0^{-1}C'' A^{q-2}L^{2-q}.
	\]
	Finally, let \(L\to\infty\). Since \(q-2>0\), the right-hand side tends to
	zero. This proves Theorem~\ref{thm: bad set}.

	\section{A local Weyl criterion}\label{sec:analytic}
	In this section, we derive a  criterion, namely, \Cref{thm:averaged}, for the validity of the Weyl asymptotics \eqref{eq:Weyl law with scalar curvature}. This criterion shows that \Cref{thm:estimate} implies \Cref{thm:main}. 
	We first explore the sublevel geometry of Ricci shrinkers in \cref{sec:sublevels}. We then study heat kernel estimates and heat kernel asymptotics in \cref{subsec: heat kernel}, and use these asymptotics to prove a local Weyl asymptotic on the good set in \cref{subsec: local weyl good}. Finally, in \Cref{subsec: Weyl bad set}, we show that the contribution of the bad set is negligible, thereby establishing \Cref{thm:averaged}.
	
	If $M$ is compact, then \eqref{eq:Weyl law with scalar curvature} is the usual Weyl law for the
	elliptic operator $H_f$ on a compact manifold.  Hence we assume from
	now on that $M$ is noncompact.

	\subsection{Sublevel set geometry}\label{sec:sublevels}

	The weighted Laplacian $L_f$ contains a first-order drift term and acts on a
	weighted $L^2$ space, while the Weyl-law machinery that we will use is most
	naturally formulated for Schr\"odinger operators on the ordinary Riemannian
	measure.  
	
	The first step is therefore to conjugate $L_f$.  A direct computation shows that the map
	\[
	U:L^2(M,e^{-f}\dd\mathrm{vol}_g)\longrightarrow
	L^2(M,\dd\mathrm{vol}_g),
	\qquad Uu=e^{-f/2}u,
	\]
	is unitary.  Let $H_f:=UL_fU^{-1}$. Then $H_f$ is a Schr\"odinger operator.
	That is,
	\be\label{eq:conjugation}
	H_f=\Delta+V_f,
	\qquad
	V_f=\frac14(f+R-n).
	\ee
	This is because, for $v\in C_c^\infty(M)$, the product rule 
	gives
	\[
	e^{-f/2}L_f(e^{f/2}v)
	=\Delta v+
	\left(\frac14|\nabla f|^2+\frac12\Delta f\right)v.
	\]
	Taking the trace of \eqref{eq:shrinker} in our sign convention yields
	$\Delta f=R-n/2$.  Since $|\nabla f|^2=f-R$, this proves
	\eqref{eq:conjugation}.

	After the conjugation above, it suffices to study Weyl law for Schr\"odinger operator
	$
	H_f$
	on $(M,g).$
	We will give quantitative control
	of three geometric objects: the size $\sigma(\l)$ of the classically allowed region
	$\{V_f<\lambda\}$, the phase integral $\Phi_f(\l)$ over that region, and the thin ``annular region” where $V_f$ is close to $\lambda$.
	
	Set
	\[
	\Omega_s=\{f<s\}\quad\text{and}\quad
	S(s)=\int_{\Omega_s}R\dd\mathrm{vol}_g.
	\]
	
	The following simple fact is from Cao-Zhou, and will be used later. 
	
	\begin{prop}[$\S$3 in Cao-Zhou \cite{CZ}]
		\label{prop:coarea}
		For a.e. $s>0$,
		\be\label{eq:coarea}
		sF'(s)-S'(s)=\frac n2F(s)-S(s).
		\ee
		Moreover,
		\be\label{eq:S-bound}
		0\leq S(s)\leq\frac n2F(s),
		\ee
		and whenever $h>0$ and
		$\l-\frac n2-(\frac n2+1)h>0$,
		\be\label{eq:shell}
		\frac{F(\l)-F(\l-h)}{F(\l)}
		\leq\frac{n(1+h)}{2(\l-h)}.
		\ee
		In particular, for every fixed $A>1$ there are constants $C_A,s_A>0$
		such that
		\be\label{eq:doubling}
		F(As)\leq C_AF(s),\qquad\forall s\geq s_A.
		\ee
	\end{prop}
	
	\begin{proof}
		At a regular value $s$, coarea and the divergence theorem give
		\begin{align*}
			sF'(s)-S'(s)
			&=\int_{\partial \Omega_s}\frac{s-R}{|\nabla f|}\dd A
			=\int_{\partial \Omega_s}|\nabla f|\dd A\\
			&=\int_{\Omega_s}\operatorname{div}(\nabla f)\dd\mathrm{vol}_g
			=\frac n2F(s)-S(s).
		\end{align*}
		Here we used $0\leq s-R=|\nabla f|^2$ on $\partial \Omega_s$ and
		$\operatorname{div}\nabla f=n/2-R$.  The non-negativity of $|\nabla f|$ and $R$ and the equation above prove \eqref{eq:S-bound}.

		Integrating \eqref{eq:coarea} from $a$ to $b=a+h$ yields
		\be\label{eq:integrated-coarea}
		bF(b)=aF(a)+\left(\frac n2+1\right)\int_a^bF(s)\dd s
		-\int_a^bS(s)\dd s+S(b)-S(a).
		\ee
		Using monotonicity of $F$, nonnegativity of $S$, and
		\eqref{eq:S-bound}, we obtain
		\[
		\left[b-\frac n2-\left(\frac n2+1\right)h\right]F(b)
		\leq aF(a).
		\]
		Taking $b=\l$ and $a=\l-h$ gives \eqref{eq:shell}.  
	\end{proof}
	
	Recall that
	\[
	\sigma(\lambda):=\Vol_g\{V_f<\lambda\}.
	\]
	
	The  potential $V_f$ contains the scalar-curvature term $R$, so
	its sublevel sets are not the same as those of $f$.  However, we will see that the extra term $R$ changes the volume of
	sublevel set by an $o(1)$ amount. 
	
	\begin{prop}\label{prop:potential-sublevels}
		As $\lambda\to\infty$,
		\be\label{eq:sigma-F}
		\sigma(\lambda)
		=F(4\lambda+n)\bigl(1+O(\lambda^{-1/2})\bigr).
		\ee
	\end{prop}
	
	\begin{proof}
		
		Note that $\{V_f<\l\}=\{f+R<4\l+n\}$. We set $\tilde\l=4\lambda+n$ and let
		$
		h:=\sqrt{\tilde{\l}}.
		$
		Since $R\geq0$,
		$\{f+R<{\tilde{\l}}\}\subset \Omega_{\tilde{\l}}$.  A point of $\Omega_{{\tilde{\l}}-h}$ outside
		$\{f+R<{\tilde{\l}}\}$ satisfies $R\geq h$.  Therefore by \Cref{prop:coarea},
		\begin{align*}
			0\leq F({\tilde{\l}})-\sigma(\lambda)
			&\leq F({\tilde{\l}})-F({\tilde{\l}}-h)
			+\Vol_g(\Omega_{{\tilde{\l}}-h}\cap\{R\geq h\})\\
			&\leq \frac{n(1+h)}{2({\tilde{\l}}-h)}F({\tilde{\l}})+\frac{S({\tilde{\l}})}h
			=O({\tilde{\l}}^{-1/2})F({\tilde{\l}}),
		\end{align*}
		
	\end{proof}
	
	We can now give a direct proof of \Cref{thm:main} under the bounded geometry assumption, without using \Cref{thm:estimate}.

	\begin{thm}
		\label{thm:main-bounded}
		Let $(M^n,g,f)$ be a complete noncompact Ricci shrinker with bounded geometry. That is, there exist $C_k,\delta>0$ such that
		\[
		|\nabla^k \Rm_g|\leq C_k
		\quad\text{and}\quad
		\inj_g\geq \delta,\quad k=0,1,2.
		\]
		Then
		\[
		N_f(\lambda)
		\sim
		(2\pi)^{-n}\omega_n
		\int_M(\lambda-V_f)_+^{n/2}\,\dvol_g,
		\quad\l\to+\infty.
		\]
	\end{thm}
	
	\begin{proof}
		It follows from \Cref{prop:coarea} and \Cref{prop:potential-sublevels} that
		$V_f$ satisfies the doubling condition:$$\sigma(2\l)\leq C\sigma(\l),\quad \text{$\l$ is sufficiently large}.$$ The theorem then
		follows from \cite[Proposition A.1 and Theorem 1.10]{WeylDY} together with the fact that
		$|\nabla f|^2\leq f$ and  $|\nabla R|$ is bounded.
	\end{proof}

	Define the comparison phase integral
	\[
	\Psi_f(\lambda)
	:=(2\pi)^{-n}\omega_n
	\int_M\left(\lambda-\frac{f-n}{4}\right)_+^{n/2}
	\dd\mathrm{vol}_g.
	\]
	
	The preceding proposition compares the \emph{volumes} of the two kinds of
	sublevel sets.  
	We next extend this comparison to the full phase-space integrals
	$\Phi_f(\l)$ and $\Psi_f(\l)$.

	\begin{prop}
		\label{prop:phase-comparison}
		For every complete noncompact shrinker,
		\be\label{eq:phase-comparison}
		\Phi_f(\lambda)\sim\Psi_f(\lambda),\quad \l\to+\infty.
		\ee
		Furthermore, for all sufficiently large $\lambda$,
		\be\label{eq:phase-lower}
		\Phi_f(\lambda)\geq c\lambda^{n/2}F(\lambda).
		\ee
		
	\end{prop}
	
	\begin{proof}
		
		Put $p=n/2$. 
		For any proper smooth potential $W$ bounded from below, Tonelli's theorem gives
		\be\label{eq: Tonelli}
		\int_M(\lambda-W)_+^p\dd\mathrm{vol}_g
		=p\int_{\inf W}^{\lambda}(\lambda-s)^{p-1}
		\Vol_g\{W<s\} \dd s.
		\ee
		By \Cref{prop:potential-sublevels}, for any $\ep>0$, there exists
		$\l_0=\l_0(\ep)$ such that, whenever $s>\l_0$,
		\[		(1-\ep)F(4s+n)\leq \sigma(s)\leq (1+\ep)F(4s+n).
		\]
		Moreover, since
		\[
		\lim_{\l\to\infty}\int_{\l_0}^{\lambda}(\lambda-s)^{p-1}
		\Vol_g\{W<s\} \dd s=\infty
		\]
		for $W=V_f$ or $W=\tilde{V}_f$, there exists
		$\l_1=\l_1(\ep)>0$ such that, whenever $\l\geq\l_1$,
		\[\ba
		p\int_{\inf V_f}^{\lambda_0}(\lambda-s)^{p-1}
		\sigma(s) \dd s
		&\leq \ep p\int_{\l_0}^{\lambda}(\lambda-s)^{p-1}
		\sigma(s) \dd s,\\
		p\int_{\inf \tilde{V}_f}^{\lambda_0}(\lambda-s)^{p-1}
		F(4s+n) \dd s
		&\leq \ep p\int_{\l_0}^{\lambda}(\lambda-s)^{p-1}
		F(4s+n) \dd s.
		\ea\]
		Thus, whenever $\l\geq\l_1$,
		\[\ba
		&\quad\int_M(\lambda-V_f)_+^p\dd\mathrm{vol}_g
		=p\int_{\inf V_f}^{\lambda}(\lambda-s)^{p-1}
		\sigma(s) \dd s\\
		&\leq (1+\ep)p\int_{\l_0}^{\lambda}(\lambda-s)^{p-1}
		\sigma(s) \dd s\leq (1+\ep)^2p\int_{\l_0}^{\lambda}(\lambda-s)^{p-1}
		F(4s+n) \dd s\\
		&\leq \frac{(1+\ep)^2}{1-\ep}
		p\int_{\inf \tilde{V}_f}^{\lambda}(\lambda-s)^{p-1}
		F(4s+n) \dd s=\frac{(1+\ep)^2}{1-\ep}
		\int_M(\lambda-\tilde{V}_f)_+^p\dd\mathrm{vol}_g.
		\ea\]
		Therefore, for $\l\geq\l_1$,
		\[
		\Phi_f(\l)\leq \frac{(1+\ep)^2}{1-\ep}\Psi_f(\l).
		\]
		Similarly, one obtains, for $\l\geq\l_1$,
		\[
		\Phi_f(\l)\geq \frac{(1-\ep)^2}{1+\ep}\Psi_f(\l),
		\]
		which proves \eqref{eq:phase-comparison}.
		
		To prove \eqref{eq:phase-lower}, by \eqref{eq: Tonelli}, for all sufficiently
		large $\l$, we have
		\[\ba
		&\quad\int_M(\lambda-V_f)_+^p\dd\mathrm{vol}_g
		\geq p\int_{\l/4}^{\lambda/2}(\lambda-s)^{p-1}
		\sigma(s) \dd s\geq p2^{-p+1}\l^p\sigma(\l/4)
		\geq C'\l^pF(\l).
		\ea\]

	\end{proof}
	
	The following estimate is also needed.  
	
	\begin{lem}\label{lem:turning-layer}
		For every sufficiently small fixed $\delta>0$,
		\[
		\Vol_g\{(1-\delta)\lambda<V_f\leq(1+\delta)\lambda\}
		\leq C(\delta+\lambda^{-1/2})F(\lambda)
		\]
		for all sufficiently large $\lambda$.  The constant is uniform for
		$\delta$ in a fixed sufficiently small interval.
	\end{lem}
	
	\begin{proof}
		This follows from \eqref{eq:sigma-F} and 
		\eqref{eq:shell} easily.
	\end{proof}

	\subsection{Heat kernel estimates and heat kernel asymptotics}\label{subsec: heat kernel}

	\begin{prop}[Gaussian heat-kernel bound]
		\label{prop:heat-kernel}
		Let $n\geq3$ and $K_f(t,x,y)$ denote the heat kernel of
		$$
		H_f=\Delta+V_f,
		\qquad
		V_f=\frac14(f+R-n).
		$$
		Then there exist a constant $C>0$ such that for all $x,y\in M$ and $0<t\leq1$,
		\be
		\label{eq:global-gaussian}
		0\leq K_f(t,x,y)
		\leq
		Ct^{-n/2}
		\exp\left(-\frac{d_g(x,y)^2}{8t}\right).
		\ee
	\end{prop}
	
	\begin{proof}The proof is standard.
		For any elliptic operator $L$, let $K_L$ denote its heat kernel.
		
		Li--Wang's Sobolev inequality \cite[Theorem 1.1]{LW20} states that 
		\[
		\|u\|_{2n/(n-2)}^2
		\leq C\int_M(4|\nabla u|^2+Ru^2)\dd\mathrm{vol}_g,
		\qquad u\in C_c^\infty(M).
		\]
		For $A=4\Delta+R$, H\"older interpolation gives the Nash inequality
		$
		\|u\|_2^{2+4/n}
		\leq C\langle Au,u\rangle\|u\|_1^{4/n}.
		$
		The standard semigroup argument yields (c.f. \cite{varopoulos1985hardy})
		$K_A(t,x,x)\leq Ct^{-n/2}$.
		Grigor'yan's on-diagonal-to-off-diagonal argument
		\cite[Theorem~1.1]{GrigoryanGaussian}
		then gives
		\be
		\label{eq:A-gaussian}
		K_{A/4}(t,x,y)
		\leq
		Ct^{-n/2}
		\exp\left(-\frac{d_g(x,y)^2}{8t}\right),
		\qquad 0<t\leq1.
		\ee
		Indeed, Grigor'yan's argument is based on integral estimates for positive
		subsolutions and applies without change to the heat equation associated
		with $4\Delta+R$, since $R\geq0$.
		
		Since
		$
		4H_f+n=A+f
		$
		and $f\geq0$, it follows from the maximum principle and \eqref{eq:A-gaussian} that
		\[
		K_{H_f}(t,x,y)\leq Ct^{-n/2}
		\exp\left(\frac{nt}{4}-\frac{d_g(x,y)^2}{8t}\right),
		\qquad 0<t\leq1.
		\]
	\end{proof}
	As a direct consequence of \Cref{prop:heat-kernel}, we obtain the following.
	\begin{prop}
		\label{prop:projector-bound}
		Let
		$
		P_\lambda=\mathbf 1_{(-\infty,\lambda)}(H_f)
		$
		and let $e_\lambda(x,y)$ be its integral kernel.  Then there exists
		$C>0$ such that
		\be
		\label{eq:projector-bound}
		e_\lambda(x,x)\leq C\lambda^{n/2}
		\ee
		for every $x\in M$ and $\lambda\geq1$.
	\end{prop}
	
	\begin{proof}
		Writing the heat kernel and the spectral projector in terms of an
		orthonormal eigenbasis of $H_f$, for every $t>0$,
		
		$$
		e_\lambda(x,x)
		=
		\sum_{\lambda_j<\lambda}|\phi_j(x)|^2
		\leq
		e^{t\lambda}
		\sum_j e^{-t\lambda_j}|\phi_j(x)|^2
		=
		e^{t\lambda}K_f(t,x,x).
		$$
		By Proposition~\ref{prop:heat-kernel},
		$
		K_f(t,x,x)\leq Ct^{-n/2},
		0<t\leq1.
		$
		Taking $t=\lambda^{-1}$ proves \eqref{eq:projector-bound}
		
	\end{proof}

	\begin{lem}
		\label{lem:local-heat-stability}
		Let $(M^n,g)$ be complete, and let
		\[
		H=\Delta+W,
		\qquad
		W\in C^\infty(M),
		\]
		where $W$ is bounded from below.
		
		Assume that its heat kernel satisfies the Gaussian estimate
		\be\label{eq:Gaussian-assumption}
		|K_H(t,y,z)|
		\leq C_0t^{-n/2}
		\exp\left(-\frac{d_g(y,z)^2}{C_0t}\right),
		\qquad 0<t\leq1.
		\ee
		Fix $x\in M$ and $L\geq2$, and suppose that
		\[
		\inf_{B_g(x,2L)}\inj_g\geq4L.
		\]
		Set
		\[
		\Theta_L
		:=
		\sum_{k=0}^2
		L^{2+k}
		\sup_{B_g(x,2L)}|\nabla^k\Rm_g|,
		\qquad
		\Omega_L
		:=
		\osc_{B_g(x,2L)}W.
		\]
		There exist $\epsilon_0,c,C>0$, depending only on $n$, $C_0$,
		and an upper bound for $\sup_{y\in B_g(x,2L)}|W(y)|$, such that if
		$\Theta_L\leq\epsilon_0$, then for every $0<t\leq1$,
		\be\label{eq:local-heat-stability}
		\begin{aligned}
			&
			\left|
			K_H(t,x,x)
			-(4\pi t)^{-n/2}e^{-tW(x)}
			\right|
			\leq
			Ct^{-n/2}
			\left(
			\frac{t}{L^2}\Theta_L
			+t\Omega_L
			+\exp\left(-\frac{cL^2}{t}\right)
			\right).
		\end{aligned}
		\ee
	\end{lem}
	
	\begin{proof}
		The proof follows the local parametrix argument in
		\cite[Propositions~3.15--3.16 and Theorem~3.17]{BDY};
		we keep track of the dependence on the local geometry and on the
		oscillation of the potential.
		
		Since $\inj_g\geq4L$ on $B_g(x,2L)$ and $\Theta_L$ is sufficiently
		small, normal coordinates centered at $x$ are available on
		$B_g(x,2L)$.  If
		$
		G(x,y):=\det(g_{ij}(y))
		$
		in these coordinates, the Jacobi-field estimates and their
		differentiated versions give
		\be\label{eq:G-small}
		\left|\Delta_yG^{-1/4}(x,y)\right|
		\leq
		C L^{-2}\Theta_L,
		\qquad y\in B_g(x,L).
		\ee
		Here and below $C$ depends only on the quantities indicated in the
		statement.
		
		Let $\chi\in C_c^\infty(B_g(x,L))$ satisfy
		\[
		\chi=1\quad\text{on }B_g(x,L/2),
		\qquad
		|\nabla\chi|\leq CL^{-1},
		\qquad
		|\Delta\chi|\leq CL^{-2},
		\]
		and consider the zeroth-order parametrix
		\[
		Q(t,x,y)
		:=
		\chi(y)(4\pi t)^{-n/2}
		e^{-tW(x)}
		\exp\left(-\frac{d_g(x,y)^2}{4t}\right)
		G(x,y)^{-1/4}.
		\]
		As in \cite[Proposition~3.15]{BDY}, before introducing the cutoff one
		has
		\[\ba
		&\quad(\partial_t+H)_y
		\left[
		(4\pi t)^{-n/2}
		e^{-tW(x)}
		e^{-d_g(x,y)^2/(4t)}
		G(x,y)^{-1/4}
		\right]\\
		&=
		(4\pi t)^{-n/2}
		e^{-tW(x)}
		e^{-d_g(x,y)^2/(4t)}
		\left(
		\Delta_yG^{-1/4}(x,y)
		+(W(y)-W(x))G(x,y)^{-1/4}
		\right).\ea
		\]
		Hence \eqref{eq:G-small} and the definition of $\Omega_L$ give, 	on $B_g(x,L/2)$,
		\[
		|(\partial_t+H)Q(t,x,y)|
		\leq
		Ct^{-n/2}e^{-d_g(x,y)^2/(8t)}
		\left(
		L^{-2}\Theta_L+\Omega_L
		\right).
		\]
		The terms produced by differentiating the cutoff are supported in
		$
		C_L:= B_g(x,L)\setminus B_g(x,L/2).
		$
		Using
		$
		d_g(x,y)\geq L/2
		$
		on this set, the polynomial factors in $L^{-1}$ and $t^{-1}$ can be
		absorbed into the Gaussian, yielding
		\[
		|(\partial_t+H)Q(t,x,y)|\big|_{C_L}
		\leq
		Ct^{-n/2}
		\exp\left(-\frac{cL^2}{t}\right).
		\]

		Finally, Duhamel's formula gives
		\[
		K_H(t,x,x)-Q(t,x,x)
		=
		-\int_0^t\int_M
		K_H(s,x,y)
		(\partial_t+H)Q(t-s,x,y)
		\,\dvol_g(y)\,ds.
		\]
		Combining \eqref{eq:Gaussian-assumption} with the preceding estimates
		and using the standard convolution estimate for two Gaussians gives
		\[
		|K_H(t,x,x)-Q(t,x,x)|
		\leq
		Ct^{-n/2}
		\left(
		\frac{t}{L^2}\Theta_L
		+t\Omega_L
		+e^{-cL^2/t}
		\right).
		\]
		Since $\chi(x)=1$ and $G(x,x)=1$,
		$
		Q(t,x,x)=(4\pi t)^{-n/2}e^{-tW(x)},
		$
		which proves \eqref{eq:local-heat-stability}.
	\end{proof}
	
	\subsection{Local Weyl asymptotics on good sets}\label{subsec: local weyl good}
	We now establish the following local Weyl asymptotics on the good set.
	
	\begin{lem}\label{lem:local-density}
		Fix $\delta\in(0,1)$.  Suppose that we have a family of $G_\lambda\subset M$ such that
		\begin{enumerate}[(1)]
			\item $
			G_\lambda\subset\{V_f\leq(1-\delta)\lambda\};
			$
			\item $
			\inf_{x\in G_\lambda}\sqrt\lambda\,\rho(x)\longrightarrow\infty,\quad\lambda\to\infty.
			$
		\end{enumerate}
		Then, uniformly for $x\in G_\lambda$,
		\[
		e_\lambda(x,x)
		=(2\pi)^{-n}\omega_n(\lambda-V_f(x))^{n/2}+o(\lambda^{n/2}).
		\]
	\end{lem}

	\begin{proof}
		We argue by contradiction. Suppose that the conclusion is not uniform.
		Then there exist $\lambda_j\to\infty$ and $x_j\in G_{\lambda_j}$ such
		that for some $\delta>0$
		\[
		\lambda_j^{-n/2}
		\left|
		e_{\lambda_j}(x_j,x_j)
		-
		(2\pi)^{-n}\omega_n
		\bigl(\lambda_j-V_f(x_j)\bigr)^{n/2}
		\right|\ge\delta.
		\]
		Set
		\[
		g_j:=\lambda_jg,
		\qquad
		W_j:=\lambda_j^{-1}V_f,
		\qquad
		m_j:=\sqrt{\lambda_j}\,\rho(x_j).
		\]
		By assumption, $m_j\to\infty$. After passing to a subsequence, we may
		assume that
		\[
		W_j(x_j)\longrightarrow v
		\qquad\text{for some }v\in[0,1-\delta].
		\]
		By the condition (2),  for every fixed $L<\infty$ and all sufficiently
		large $j$,
		\[
		\bigl|\nabla^k\Rm_{g_j}\bigr|
		\leq C m_j^{-2-k},
		\qquad 0\leq k\leq2,
		\qquad
		\inj_{g_j}\geq m_j\quad\text{	on $B_{g_j}(x_j,2L)$.}
		\]
		The rescaled potential becomes constant on every fixed ball. Indeed,
		since
		\[
		V_f(x_j)\leq(1-\delta)\lambda_j
		\]
		and $R\geq0$, one has $f(x_j)=O_\delta(\lambda_j)$. Using
		$|\nabla\sqrt f|\leq\frac12$ together with
		$
		|\nabla R|\leq C\rho(x_j)^{-3}
		$
		on a slightly smaller regularity ball, we obtain, for every fixed
		$L<\infty$,
		\be
		\osc_{B_{g_j}(x_j,L)}W_j
		\leq
		C_{\delta,L}\lambda_j^{-1}
		+
		C_Lm_j^{-3}
		\longrightarrow0.
		\ee
		Let
		\[
		H_j:=\Delta_{g_j}+W_j=\lambda_j^{-1}H_f.
		\]
		Under the rescaling $g_j=\lambda_jg$, the global Gaussian estimate
		\eqref{eq:global-gaussian} becomes
		\[
		K_{H_j}(t,x,y)
		\leq
		Ct^{-n/2}
		\exp\left(-\frac{d_{g_j}(x,y)^2}{8t}\right),
		\qquad 0<t\leq1,
		\]
		with a constant $C$ independent of $j$.

		Applying Lemma~\ref{lem:local-heat-stability}, using the bounds above and
		the fact that $W_j(x_j)\to v$, and then letting $j\to\infty$ followed by
		$L\to\infty$, we obtain
		\[
		K_{H_j}(t,x_j,x_j)
		\longrightarrow
		(4\pi t)^{-n/2}e^{-tv},
		\qquad \text{for any fixed } t\in(0,1).
		\]
		Equivalently,
		\be\label{eq: convergence of heat kernel in sequence}
		\lambda_j^{-n/2}
		K_{H_f}\left(\frac{t}{\lambda_j},x_j,x_j\right)
		\longrightarrow
		(4\pi t)^{-n/2}e^{-tv}, \forall t\in(0,1).
		\ee
		Define positive locally finite measures on $[0,\infty)$ by
		\[
		\mu_j([0,s))
		:=
		\lambda_j^{-n/2}e_{\lambda_js}(x_j,x_j).
		\]
		By the spectral theorem,
		\[
		\int_{[0,\infty)}e^{-ts}\,d\mu_j(s)
		=
		\lambda_j^{-n/2}
		K_{H_f}\left(\frac{t}{\lambda_j},x_j,x_j\right).
		\]
		Therefore, \eqref{eq: convergence of heat kernel in sequence} gives
		\be
		\int_{[0,\infty)}e^{-ts}\,d\mu_j(s)
		\longrightarrow
		(4\pi t)^{-n/2}e^{-tv}.
		\ee
		
		Let $\mu_v$ be the positive measure determined by
		\[
		\mu_v([0,s))
		=
		(2\pi)^{-n}\omega_n(s-v)_+^{n/2}.
		\]
		Its Laplace transform is
		\[
		\int_{[0,\infty)}e^{-ts}\,d\mu_v(s)
		=
		(4\pi t)^{-n/2}e^{-tv}.
		\]
		Moreover, \Cref{prop:projector-bound} gives, for every $S<\infty$,
		$
		\sup_j\mu_j([0,S])<\infty.
		$
		The Gaussian heat-kernel estimate also gives, for every fixed $t\in(0,1)$,
		\[
		\sup_j
		\int_{[0,\infty)}e^{-ts}\,d\mu_j(s)
		=
		\sup_j
		\lambda_j^{-n/2}
		K_{H_f}\left(\frac{t}{\lambda_j},x_j,x_j\right)
		<\infty.
		\]
		Consequently, the sequential form of the semiclassical Tauberian
		theorem \cite[Theorem~1.7]{WeylDY}, applied with $\alpha=0$, yields
		\[
		\lambda_j^{-n/2}e_{\lambda_j}(x_j,x_j)
		\longrightarrow
		(2\pi)^{-n}\omega_n(1-v)^{n/2}.
		\]
		On the other hand,
		\[
		\lambda_j^{-n/2}
		(2\pi)^{-n}\omega_n
		\bigl(\lambda_j-V_f(x_j)\bigr)^{n/2}
		=
		(2\pi)^{-n}\omega_n
		\bigl(1-W_j(x_j)\bigr)^{n/2}
		\longrightarrow
		(2\pi)^{-n}\omega_n(1-v)^{n/2}.
		\]
		This contradicts the choice of $(\lambda_j,x_j)$ and proves that,
		uniformly for $x\in G_\lambda$,
		\[
		e_\lambda(x,x)
		=
		(2\pi)^{-n}\omega_n
		\bigl(\lambda-V_f(x)\bigr)^{n/2}
		+
		o(\lambda^{n/2}).
		\]
	\end{proof}
	
	\subsection{Integral of $e_\l(x,x)$ over the bad set}\label{subsec: Weyl bad set}

	We now analyze the integral of $e_\l(x,x)$ over the bad set. We first establish the following estimate.

	\begin{prop}\label{prop:high-R}
		For every fixed $C_0,\epsilon>0$,
		\be\label{eq:high-R}
		\int_{\{f<C_0\lambda,\ R\geq\epsilon\lambda\}}
		e_\lambda(x,x)\dd\mathrm{vol}_g(x)
		=o(\Phi_f(\lambda)).
		\ee
	\end{prop}
	
	\begin{proof}
		By \eqref{eq:S-bound},
		\[
		\Vol_g\{f<C_0\lambda,\ R\geq\epsilon\lambda\}
		\leq\frac{S(C_0\lambda)}{\epsilon\lambda}
		\leq\frac{n}{2\epsilon\lambda}F(C_0\lambda).
		\]
		Combine this with \eqref{eq:projector-bound},
		\eqref{eq:doubling}, and \eqref{eq:phase-lower}.  The ratio of the left
		side of \eqref{eq:high-R} to $\Phi_f(\lambda)$ is
		$O((\epsilon\lambda)^{-1})$.
	\end{proof}

	\begin{lem}\label{lem:agmon}
		Fix $\delta\in(0,1)$.  There are $c,C>0$ such that, for every sufficiently
		large $\lambda$ and every  eigenfunction
		$H_fu=\nu u$ with $\nu\leq\lambda$,
		\be\label{eq:agmon}
		\int_{\{f\geq4(1+\delta/2)\lambda+n\}}|u|^2\dd\mathrm{vol}_g
		\leq C\delta^{-1}e^{-c\delta^{3/2}\lambda}\int_M|u|^2\dvol_g.
		\ee
		Consequently,
		\[
		\int_{\{f\geq4(1+\delta/2)\lambda+n\}}e_\l(x,x)\dd\mathrm{vol}_g
		=o(\Phi_f(\l)).
		\]

	\end{lem}
	
	\begin{proof}
		
		The proof is a standard Agmon estimate argument; see \cite[Lemma 3.1]{DY2020cohomology} or \cite{agmon2014lectures} for details. We only outline the proof here.
		
		Set
		\[
		a_\lambda=4(1+\delta/4)\lambda+n,\qquad
		\gamma^2=\frac{\delta}{32(1+\delta)},\qquad
		\phi=\gamma(f-a_\lambda)_+.
		\]
		We first use the bounded Lipschitz truncations
		\[
		\phi_N=\min\{\phi,N\}.
		\]
		Since the gradient of $\phi_N$ is supported in a compact $f$-sublevel set,
		$e^{\phi_N}u\in\mathrm{Dom}(H_f)$. Using $e^{2\phi_N}u$ as a test function
		and integrating by parts, we obtain
		\be\label{eq:agmon-identity}
		0=\int_M\Big(
		|\nabla(e^{\phi_N}u)|^2+
		(V_f-\nu-|\nabla\phi_N|^2)e^{2\phi_N}|u|^2
		\Big)\dd\mathrm{vol}_g.
		\ee
		On $\{f\geq a_\lambda\}$, the inequalities
		$V_f\geq(f-n)/4$ and $|\nabla f|^2\leq f$ imply
		\[
		V_f-\nu-|\nabla\phi_N|^2\geq c\delta\lambda.
		\]
		On $\{f<a_\lambda\}$, the same coefficient is bounded below by
		$-(\lambda+n/4)$. Dropping the nonnegative gradient term in
		\eqref{eq:agmon-identity}, we obtain
		\[
		\int_{\{f\geq a_\lambda\}}e^{2\phi_N}|u|^2
		\dd\mathrm{vol}_g
		\leq C\delta^{-1}\int_{\{f<a_\lambda\}}|u|^2\dd\mathrm{vol}_g
		\leq C\delta^{-1}\int_M|u|^2\dd\mathrm{vol}_g.
		\]
		Since the integrands are nonnegative, monotone convergence applies as
		$N\to\infty$. On $\{f\geq4(1+\delta/2)\lambda+n\}$, we have
		$\phi\geq\gamma\delta\lambda$, which proves \eqref{eq:agmon}.

		Assume that $\l$ is sufficiently large and sum \eqref{eq:agmon} over $\nu<\lambda$. Fei He's estimate \cite[Theorem 1.2]{He} implies
		$
		N_f(\lambda)\leq C(1+\lambda)^n,
		$
		which contributes only a polynomial factor, while \eqref{eq:phase-lower}, \cite[Theorem 1.6]{MunteanuWang2012}, and \cite[Theorem 1.1]{CZ} imply
		$
		\Phi_f(\l)\geq c\l^{\frac{n+1}{2}}.
		$
		The exponential decay in \eqref{eq:agmon} therefore proves the last assertion.
		
	\end{proof}

	We are now ready to prove the main result of this section.

	\begin{thm}\label{thm:averaged}
		Let $(M^n,g,f),n\geq3$ be a complete noncompact shrinker.
		Assume that for every fixed $A,C>0$,
		\be\label{eq:averaged-hypothesis}
		\Vol_g\{f<C\lambda,\ \sqrt\lambda\,\rho(x)\leq A\}
		=o(F(\lambda)).
		\ee
		Then the Weyl law \eqref{eq:Weyl law with scalar curvature} holds.
	\end{thm}
	
	\begin{proof}
		Fix a sufficiently small $\delta\in(0,1)$.  By 
		\eqref{eq:averaged-hypothesis}, there exists $A_\lambda\to\infty$ such
		that
		\be\label{eq:volume of bad}
		\Vol_g\{f<8\lambda,\ \sqrt\lambda\,\rho\leq A_\lambda\}
		=o(F(\lambda)),
		\ee
		where $8$ is chosen to satisfy
		$\{V_f\leq(1+\delta)\lambda\}\subset\{f<8\lambda\}$.
		Split $\{V_f\leq (1-\delta)\l\}$ into
		\begin{align*}
			G_\lambda
			&=\{V_f\leq(1-\delta)\lambda,\
			\sqrt\lambda\,\rho>A_\lambda\},\\
			B_\lambda
			&=\{V_f\leq(1-\delta)\lambda,\
			\sqrt\lambda\,\rho\leq A_\lambda\}.
		\end{align*}
		Lemma~\ref{lem:local-density}, Proposition
		\ref{prop:projector-bound}, \eqref{eq:doubling}, and
		\eqref{eq:phase-lower} imply
		\be\label{eq:integrated-local-density}
		\int_{G_\lambda\cup B_\lambda}e_\lambda(x,x)
		\dd\mathrm{vol}_g
		=(2\pi)^{-n}\omega_n\int_{\{V_f\leq(1-\delta)\lambda\}}
		(\lambda-V_f)^{n/2}\dd\mathrm{vol}_g
		+o(\Phi_f(\lambda)).
		\ee
		Indeed, the uniform $o(\lambda^{n/2})$ error on $G_\lambda$ is integrated over a set of volume $O(F(\lambda))$, while, by \eqref{eq:volume of bad}, both the integrals of $e_\l(x,x)$ and $(\l-V_f(x))_+^{n/2}$ over $B_\lambda$ are bounded by
		$
		o(\Phi_f(\lambda)).
		$

		By
		Lemma~\ref{lem:turning-layer} and
		Proposition~\ref{prop:projector-bound},
		\be\ba\label{eq:turning-spectral}
		\int_{\{(1-\delta)\lambda<V_f\leq(1+\delta)\lambda\}}
		e_\lambda(x,x)\dd\mathrm{vol}_g
		&\leq C(\delta+o(1))\Phi_f(\lambda);\\
		\int_{\left\{(1-\delta) \lambda<V_f<(1+\delta)\lambda\right\}}\left(\lambda-V_f\right)^{n / 2} \dvol_g   &\leq(\delta \lambda)^{n / 2} \operatorname{Vol}_g\left\{(1-\delta) \lambda<V_f \leq(1+\delta) \lambda\right\}=O(\delta) \Phi_f(\lambda).
		\ea\ee

		Finally, consider $\{V_f>(1+\delta)\lambda\}$. On the part where
		$f\leq4(1+\delta/2)\lambda+n$, we have
		\[
		R=4V_f+n-f>2\delta\lambda,
		\]
		so Proposition~\ref{prop:high-R} gives an $o(\Phi_f(\lambda))$ spectral
		contribution. The contribution from the complementary part is
		$o(\Phi_f(\lambda))$ by Lemma~\ref{lem:agmon}. Combining this with \eqref{eq:integrated-local-density} and
		\eqref{eq:turning-spectral}, and first letting $\lambda\to\infty$ and then
		$\delta\downarrow0$, we obtain the Weyl law
		\eqref{eq:Weyl law with scalar curvature}.

	\end{proof}
	
	By \Cref{thm:averaged} and \Cref{thm:estimate}, it suffices to prove \Cref{thm:main} for \(n\leq 2\). This follows from the
	classification of Ricci shrinkers in dimensions \(n\leq 2\), together with the
	Weyl law for harmonic oscillators.
	
	\bibliographystyle{plain}
	\bibliography{references}

@article{BDY,
  author        = {Braverman, M. and Dai, X. and Yan, J.},
  title         = {The Classical {W}eyl Law for {S}chr\"odinger Operators on Complete {R}iemannian Manifolds},
  year          = {2026},
  journal = {arXiv:2605.07200}
}

@article{MunteanuWang2012,
  author  = {Munteanu, O. and Wang, J.},
  title   = {Analysis of weighted {L}aplacian and applications to
             {R}icci solitons},
  journal = {Comm. Anal. Geom.},
  volume  = {20},
  number  = {1},
  pages   = {55--94},
  year    = {2012},
  doi     = {10.4310/CAG.2012.v20.n1.a3}
}

@article{MunteanuMuTaoWang,
  author  = {Munteanu, O. and Wang, M.-T.},
  title   = {The curvature of gradient {R}icci solitons},
  journal = {Math. Res. Lett.},
  volume  = {18},
  number  = {6},
  pages   = {1051--1069},
  year    = {2011},
  doi     = {10.4310/MRL.2011.v18.n6.a2}
}

@article{MunteanuJiapingWang,
  author  = {Munteanu, O. and Wang, J.},
  title   = {Geometry of shrinking {R}icci solitons},
  journal = {Compos. Math.},
  volume  = {151},
  number  = {12},
  pages   = {2273--2300},
  year    = {2015},
  doi     = {10.1112/S0010437X15007496}
}

@article{MunteanuWangInfinity,
  author  = {Munteanu, O. and Wang, J.},
  title   = {Structure at infinity for shrinking {R}icci solitons},
  journal = {Ann. Sci. \'Ec. Norm. Sup\'er.},
  volume  = {52},
  number  = {4},
  pages   = {891--925},
  year    = {2019},
  doi     = {10.24033/asens.2400}
}

@article{KotschwarWang,
  author  = {Kotschwar, B. and Wang, L.},
  title   = {Rigidity of asymptotically conical shrinking gradient {R}icci solitons},
  journal = {J. Differential Geom.},
  volume  = {100},
  number  = {1},
  pages   = {55--108},
  year    = {2015}
}

@article{LiLiWangStructure,
  author  = {Li, H. and Li, Y. and Wang, B.},
  title   = {On the structure of {R}icci shrinkers},
  journal = {J. Funct. Anal.},
  volume  = {280},
  number  = {9},
  pages   = {108955},
  year    = {2021},
  doi     = {10.1016/j.jfa.2021.108955}
}

@article{HuangLiWang,
  author  = {Huang, S. and Li, Y. and Wang, B.},
  title   = {On the regular-convexity of {R}icci shrinker limit spaces},
  journal = {J. Reine Angew. Math.},
  volume  = {771},
  pages   = {99--136},
  year    = {2021},
  doi     = {10.1515/crelle-2020-0021}
}

@article{LiWangKahler,
  author  = {Li, Y. and Wang, B.},
  title   = {On {K}{\"a}hler {R}icci shrinker surfaces},
  journal = {Acta Math.},
  volume  = {236},
  number  = {1},
  pages   = {1--50},
  year    = {2026},
  doi     = {10.4310/ACTA.2026.v236.n1.a1}
}

@article{CZ,
  author  = {Cao, H.-D. and Zhou, D.},
  title   = {On Complete Gradient Shrinking {R}icci Solitons},
  journal = {Journal of Differential Geometry},
  volume  = {85},
  year    = {2010},
  pages   = {175--185}
}

@article{GrigoryanGaussian,
  author  = {Grigor'yan, A.},
  title   = {Gaussian Upper Bounds for the Heat Kernel on Arbitrary Manifolds},
  journal = {J. Differential Geom.},
  volume  = {45},
  number  = {1},
  pages   = {33--52},
  year    = {1997},
  doi     = {10.4310/jdg/1214459753}
}

@book{agmon2014lectures,
  author    = {Agmon, S.},
  title     = {Lectures on Exponential Decay of Solutions of Second-Order Elliptic Equations: Bounds on Eigenfunctions of {N}-Body {S}chr\"odinger Operators},
  series    = {Mathematical Notes},
  volume    = {29},
  publisher = {Princeton University Press},
  address   = {Princeton, NJ},
  year      = {1982}
}

@article{DY2020cohomology,
  title={{W}itten deformation for noncompact manifolds with bounded geometry},
  author={Dai, X. and Yan, J.},
  journal={Journal of the Institute of Mathematics of Jussieu},
  volume={22},
  number={2},
  pages={643--680},
  year={2023},
  publisher={Cambridge University Press}
}

@article{CGT,
  author  = {Cheeger, J. and Gromov, M. and Taylor, M.},
  title   = {Finite Propagation Speed, Kernel Estimates for Functions of the {L}aplace Operator, and the Geometry of Complete {R}iemannian Manifolds},
  journal = {Journal of Differential Geometry},
  volume  = {17},
  year    = {1982},
  pages   = {15--53}
}

@article{He,
  author        = {He, F.},
  title         = {Spectral Bounds for {$f$}-{L}aplace-Type Operators with Applications to {B}etti Numbers on Gradient {R}icci Shrinkers},
  year          = {2026},
  journal = {arXiv:2607.21959}
}

@article{LW20,
  author  = {Li, Y. and Wang, B.},
  title   = {Heat Kernel on {R}icci Shrinkers},
  journal = {Calculus of Variations and Partial Differential Equations},
  volume  = {59},
  year    = {2020},
  note    = {Paper No. 194}
}

@article{LWheat,
  author  = {Li, Y. and Wang, B.},
  title   = {Heat Kernel on {R}icci Shrinkers ({II})},
  journal = {Acta Math. Sci. Ser. B (Engl. Ed.)},
  volume  = {44},
  number  = {5},
  pages   = {1639--1695},
  year    = {2024},
  doi     = {10.1007/s10473-024-0502-7}
}

@article{BamlerCompactness,
  author  = {Bamler, R. H.},
  title   = {Compactness Theory of the Space of Super {R}icci Flows},
  journal = {Invent. Math.},
  volume  = {233},
  number  = {3},
  pages   = {1121--1277},
  year    = {2023},
  doi     = {10.1007/s00222-023-01196-3}
}

@article{BamlerStructure,
  author        = {Bamler, R. H.},
  title         = {Structure Theory of Non-Collapsed Limits of {R}icci Flows},
  year          = {2020},
  journal = {arXiv:2009.03243}
}

@article{varopoulos1985hardy,
  title={{H}ardy-{L}ittlewood theory for semigroups},
  author={Varopoulos, N. Th.},
  journal={Journal of functional analysis},
  volume={63},
  number={2},
  pages={240--260},
  year={1985},
  publisher={Academic Press}
}

@article{Shi,
  author  = {Shi, W.},
  title   = {Deforming the Metric on Complete {R}iemannian Manifolds},
  journal = {Journal of Differential Geometry},
  volume  = {30},
  year    = {1989},
  pages   = {223--301}
}

@article{WeylDY,
  author        = {Dai, X. and Yan, J.},
  title         = {{W}eyl Law for {S}chr\"odinger Operators on Noncompact Manifolds, Heat Kernel, and {K}aramata--{H}ardy--{L}ittlewood Theorem},
  year          = {2025},
  journal = {arXiv:2504.15551}
}

@article{ChengZhouDrifted,
  author  = {Cheng, X. and Zhou, D.},
  title   = {Eigenvalues of the drifted {L}aplacian on complete metric measure spaces},
  journal = {Commun. Contemp. Math.},
  volume  = {19},
  number  = {1},
  pages   = {1650001},
  year    = {2017},
  doi     = {10.1142/S0219199716500012}
}

@article{LiZhangZhangEigenvalues,
  author  = {Li, C. and Zhang, H. and Zhang, X.},
  title   = {Rigidity of Eigenvalues of Shrinking {R}icci Solitons},
  journal = {J. Funct. Anal.},
  volume  = {290},
  number  = {9},
  pages   = {111386},
  year    = {2026},
  doi     = {10.1016/j.jfa.2026.111386}
}

@article{ZhangEigenvalueSplitting,
  author  = {Zhang, H.},
  title   = {Characterizing the Splitting of {R}icci Shrinkers via Eigenvalues},
  journal = {J. Geom. Anal.},
  volume  = {36},
  number  = {5},
  pages   = {166},
  year    = {2026},
  doi     = {10.1007/s12220-026-02418-9}
}

@article{ColdingMinicozziModifiedFlow,
  author  = {Colding, T. H. and Minicozzi II, W. P.},
  title   = {Eigenvalue Lower Bounds and Splitting for Modified {R}icci Flow},
  journal = {Adv. Nonlinear Stud.},
  volume  = {24},
  number  = {1},
  pages   = {178--188},
  year    = {2024},
  doi     = {10.1515/ans-2022-0083}
}

@article{ColdingMinicozziGrowth,
  author  = {Colding, T. H. and Minicozzi II, W. P.},
  title   = {Singularities of {R}icci Flow and Diffeomorphisms},
  journal = {Publ. Math. Inst. Hautes \'Etudes Sci.},
  volume  = {142},
  pages   = {75--152},
  year    = {2025},
  doi     = {10.1007/s10240-025-00157-1}
}

@article{ChengConradoPinheiroZhou,
  author        = {Cheng, X. and Conrado, F. and Pinheiro, N. and Zhou, D.},
  title         = {Eigenvalue Estimates for {S}chr\"odinger Operators on {R}icci Shrinkers},
  year          = {2026},
  journal = {arXiv:2605.23199}
}

@article{MaiOuLiouville,
  author  = {Mai, W. and Ou, J.},
  title   = {{L}iouville theorem on {R}icci shrinkers with constant scalar
             curvature and its application},
  journal = {J. Reine Angew. Math.},
  volume  = {810},
  pages   = {283--299},
  year    = {2024},
  doi     = {10.1515/crelle-2024-0021}
}

@article{WuWuPolynomialGrowth,
  author  = {Wu, J.-Y. and Wu, P.},
  title   = {Harmonic and {S}chr\"odinger functions of polynomial growth
             on gradient shrinking {R}icci solitons},
  journal = {Geom. Dedicata},
  volume  = {217},
  number  = {4},
  pages   = {75},
  year    = {2023},
  doi     = {10.1007/s10711-023-00810-1}
}

@article{HeOuHolomorphic,
  author        = {He, F. and Ou, J.},
  title         = {The Dimension of Polynomial Growth Holomorphic Functions and Forms on Gradient {K}\"ahler {R}icci Shrinkers},
  year          = {2024},
  journal = {arXiv:2401.02685}
}

@article{HeOuSections,
  author  = {He, F. and Ou, J.},
  title   = {Dimension Estimate and Existence of Holomorphic Sections
             With Polynomial Growth on Gradient {K}\"ahler {R}icci Shrinkers},
  journal = {Int. Math. Res. Not. IMRN},
  year    = {2025},
  number  = {23},
  pages   = {rnaf350},
  doi     = {10.1093/imrn/rnaf350}
}

@book{SaloffCoste2002,
  author    = {Saloff-Coste, L.},
  title     = {Aspects of {S}obolev-Type Inequalities},
  series    = {London Mathematical Society Lecture Note Series},
  volume    = {289},
  publisher = {Cambridge University Press},
  year      = {2002}
}

@article{HeinNaber,
  author  = {Hein, H.-J. and Naber, A.},
  title   = {New logarithmic {S}obolev inequalities and an
             $\varepsilon$-regularity theorem for the {R}icci flow},
  journal = {Comm. Pure Appl. Math.},
  volume  = {67},
  number  = {9},
  pages   = {1543--1561},
  year    = {2014},
  doi     = {10.1002/cpa.21474}
}
	
\end{document}